\documentclass[12pt,reqno]{amsart}
\usepackage{graphicx} 
\usepackage{amsmath,amssymb,amsthm,amsfonts}
\usepackage{enumitem}
\usepackage{comment}
\usepackage{dsfont}
\usepackage{cleveref}
\usepackage{tikz}
\usetikzlibrary{shapes.geometric, arrows}
\usepackage[a4paper, margin=1in]{geometry}

\tikzstyle{arrow} = [thick,->,>=stealth]

\tikzstyle{startstop} = [rectangle, rounded corners, 
minimum width=3cm, 
minimum height=1cm,
text centered, 
draw=black, 
fill=red!30]

\tikzstyle{io} = [rectangle, rounded corners, 
minimum width=3cm, 
minimum height=1cm,
text centered, 
draw=black, 
fill=green!30]

\tikzstyle{process} = [rectangle, 
minimum width=3cm, 
minimum height=1cm, 
text centered, 
text width=3cm, 
draw=black, 
fill=orange!30]

\newcommand{\leqnomode}{\tagsleft@true\let\veqno\@@leqno}

\DeclareMathOperator{\divg}{\mathrm{div}}
\DeclareMathOperator{\diag}{\mathrm{diag}}

\usepackage{xcolor}
\usepackage{appendix}

\newtheorem{theorem}{Theorem}
\newtheorem{lemma}[theorem]{Lemma}

\title{Non-density of nodal lines in the clamped plate problem}

\author{Alberto Enciso}
\address{Instituto de Ciencias Matem\'aticas, Consejo Superior de
  Investigaciones Cient\'\i ficas, 28049 Madrid, Spain}

\author{Josef E. Greilhuber}
\address{Stanford University, Department of Mathematics, Stanford, CA 94305, USA}
\date{\today}


\makeatletter
\@addtoreset{subsubsection}{section}
\makeatother

\numberwithin{equation}{section}

\begin{document}

\maketitle

\begin{abstract}
    We show that, in contrast to the case of Laplace eigenfunctions, the nodal set of high energy eigenfunctions of the clamped plate problem is not necessarily dense, and can in fact exhibit macroscopic ``nodal voids''. Specifically, we show that there are small deformations of the unit disk admitting a clamped plate eigenfunction of arbitrarily high frequency that does not vanish in a disk of radius $0.44$.
\end{abstract}

\section{Introduction}

In this paper we are interested in the \emph{clamped plate problem}, a classical question in mathematical physics dating back to Lord Rayleigh~\cite[Chapter VIII]{Rayleigh1877} which models the vibration of a thin metallic plate that is fixed along its edge. Mathematically, this boils down to the analysis of the fourth-order spectral problem
\begin{align}
    \label{eq:bvp1}
    \Delta^2 u = \lambda^2 u\quad &\text{ in }D,  \\
    \label{eq:bvp2}
    u = \frac{\partial u}{\partial n} = 0 \quad &\text{ on } \partial D,
\end{align}
posed on a smoothly bounded domain $D \subset \mathbb R^2$, with~${\lambda}>0$ being the (squared) vibration frequency of the clamped plate. There is an extensive literature on various aspects of this problem, including Rayleigh's conjecture~(see \cite{Hadamard1908}) that circular plates minimize the first eigenvalue among domains with fixed area~\cite{Talenti1981,Nadirashvili1995, Ashbaugh1995}, eigenvalue bounds~\cite{PPW,ChengWei2011,ChengWei2013} and Weyl asymptotics~\cite{Pleijel1950,Agmon1965}, multiplicity bounds for the circular plate~\cite{LvovskyMangoubi2022,MangoubiRosenblatt2024}, and results on the positivity of the first eigenfunction~\cite{Duffin1953}.

A particularly appealing feature of the clamped plate problem is its close relationship to the better-known \emph{membrane problem}. This corresponds to the eigenvalue problem for the {Dirichlet Laplacian},
\begin{align}
    \label{eq:bvp3}
    -\Delta u = \lambda u \quad &\text{ in }D, \\
    \label{eq:bvp4}
    u = 0 \quad &\text{ on } \partial D,
\end{align}
which models a taut membrane vibrating at (squared) frequency ${\lambda}$. In contrast to the Laplacian, the bi-Laplacian does not satisfy a maximum principle. As a result, many standard proofs of fundamental properties of Dirichlet Laplace eigenfunctions break down, and some of them have ultimately been shown to fail. Notably, it has long been known that the first clamped plate eigenfunction may change sign~\cite{Duffin1953}, much like the Dirichlet Green's function of the bi-Laplacian on a bounded domain~\cite{Duffin1949, Garabedian, Shapiro1994}, contrary to Hadamard's conjecture in his prize memoir~\cite{Hadamard1908}.

Other than that, little is known about the nodal (that is, zero) set of eigenfunctions of the clamped plate problem. In the case of Laplace eigenfunctions, perhaps the most fundamental property of the nodal set is its density. More precisely, a Laplace eigenfunction with eigenvalue~$\lambda$ must vanish on any ball of radius $c_0\lambda^{-1/2}$ contained in~$D$, where $c_0$ is an explicit constant (namely, the first positive zero of the Bessel function $J_0$). Four different proofs of this fact can be found in the book~\cite[Section 4.2]{LMP}. 

This naturally raises the question of whether nodal sets of the clamped plate problem are also dense, in the sense that eigenfunctions necessarily vanish on balls of radius $C\lambda^{-1/2}$ for some~$C$. Our main theorem asserts that this is in general not the case: even on convex domains, high-energy clamped plate eigenfunctions can exhibit macroscopic ``nodal voids''. Specifically, we prove the considerably stronger result that there exist small deformations of the unit disk admitting a clamped plate eigenfunction of arbitrarily high frequency that does not vanish on a disk of radius 0.44. A precise statement is as follows:

\begin{theorem}
    \label{thm:main}
    There exists an increasing sequence of integers $(N_k)_{k \in \mathbb N}$ and corresponding domains $(D_{k})_{k \in \mathbb N}$ in $\mathbb R^2$, which converge to the unit disk in the smooth topology, such that $D_k$ admits a clamped plate eigenfunction $u_k$ with squared frequency $\lambda_k \geq N_k^2$ with the property that $u_k > 0$ on a disk of radius 
    $$
    r_\infty \,\exp(-4 N_k^{-2/3} - (500 + 50 \log N_k) N_k^{-1})\,,
    $$
    where $r_\infty = 0.44367\dots$ is the unique solution in $(0,1)$ to the equation
    \[
    \log(r) + \sqrt{1-r^2} - \log(1 + \sqrt{1-r^2}) - r + 1 = 0.
    \]
\end{theorem}

We are able to show that the asymptotic radius $r_\infty \approx 0.44367 \ldots$ of the nodal void is optimal in the following sense.

\begin{theorem}
    \label{thm:sharpness}
    There exist $C, \varepsilon_0 > 0$ such that for every $\varepsilon \in (0,\varepsilon_0)$ and $\Phi: \mathbb D \to \mathbb R^2$ with $\|\Phi - \mathrm{id} \|_{C^{3}(\mathbb D)} \leq \varepsilon$, any eigenfunction $u$ of the clamped plate problem on $D = \Phi(\mathbb D)$ with frequency $\sqrt{\lambda} > C \frac{|\log\varepsilon|}{\varepsilon}$ changes sign on the disk of radius $(1+C\varepsilon) r_\infty$ around $0$.
\end{theorem}

The phenomena exposed by \Cref{thm:main} and \Cref{thm:sharpness} are somewhat reminiscent of Bruno and Galkowski's discovery regarding the Steklov problem, namely that for many domains in $\mathbb R^2$, there exists a fixed sub-domain on which infinitely many Steklov eigenfunctions do not vanish \cite{BrunoGalkowski2020}. In the present case, however, only one such eigenfunction is constructed per domain. It is thus natural to ask whether there are domains for which infinitely many eigenfunctions of the clamped plate problem do not have dense nodal set, and whether for any domain $D$ and any open subset $K \subseteq D$, a density-one subsequence of eigenfunctions does have nodal lines intersecting $K$, contrary to the situation encountered in the Steklov problem. The strategy developed in this paper does not seem to address these questions, at least not in an obvious way.

The main ideas of the proof of Theorem~\ref{thm:main}, as well as the core of the argument, are presented in Section~\ref{sec:proof}. Making some of these ideas precise requires a considerable number of estimates, which are deferred to later sections. With these in hand, the proof is finally completed in Section~\ref{sec:final}. Throughout, we have attempted to provide numerical bounds for the constants involved, as these bounds make precise the idea that the existence of nodal voids much larger than the characteristic length scale $\lambda_k^{-1/2}\, (\approx N_k^{-1})$ can be already observed at values of~$N_k$ that are not exceedingly large, say $N_k\approx 400$. Section~\ref{sec:sharpness} is dedicated to the proof of \Cref{thm:sharpness}, which leans on a special case of a currently unpublished lemma of Decio, Malinnikova and Nazarov, presented with proof in Appendix~\ref{sec:DMN}.

\section{Proof of Theorem 1: Main ideas}\label{sec:proof}

The gist of the proof of~\Cref{thm:main} is to perturb the unit disk very slightly and to track the behavior of a particular eigenfunction of the clamped plate problem under this perturbation. In this section we present the general framework in which we implement this idea and outline the main ingredients of the proof. Along the way, we will need several conceptually simple lemmas, e.g.\ bounds on Bessel functions and concrete derivative estimates on families of eigenfunctions obtained from perturbation theory. The proofs of these lemmas have been relegated to later sections.

\subsubsection{Initial observations}
Let $D$ be a smoothly bounded domain, and let $u$ be a solution to \eqref{eq:bvp1}-\eqref{eq:bvp2} in $D$. Let $v = \frac{1}{2\lambda}(\Delta u - \lambda u)$ and $w = \frac{1}{2\lambda}(\Delta u + \lambda u)$, so that $u = w-v$. Then $\Delta v + \lambda v = 0 = \Delta w - \lambda w$, i.e.\ $v$ solves the Helmholtz equation and $w$ solves the screened Poisson equation. Hence, we shall call $v$ the \emph{Helmholtz component} and $w$ the \emph{SP component}. The boundary condition $u = \frac{\partial u}{\partial n} u = 0$ is equivalent to $w - v = \frac{\partial w}{\partial n} - \frac{\partial v}{\partial n} = 0$. In fact, the coupled boundary value problem 
\begin{align}
    \label{eq:bvp5}
    \Delta v + \lambda v = \Delta w - \lambda w = 0 \quad &\text{ in }  D, \\
    \label{eq:bvp6}
    w - v = \frac{\partial w}{\partial n} - \frac{\partial v}{\partial n} = 0 \quad &\text{ on } \partial D,
\end{align}
is equivalent to \eqref{eq:bvp1}-\eqref{eq:bvp2}.

The Helmholtz component certainly exhibits the typical oscillatory behavior of a Laplace eigenfunction, vanishing on any ball of radius $c_0\lambda^{-1/2}$. The SP component, which is the key ingredient in our proof, does not oscillate; instead, it falls off exponentially away from the boundary. One might expect that, away from the boundary, a high-energy clamped eigenfunction~$u$ coincides with the Helmholtz component~$v$ up to a small error corresponding to~$w$, which would heuristically lead to the density of the nodal set over balls of radius~$c_0\lambda^{-1/2}$. However, as we shall see, this picture is not completely accurate.

The proof of \Cref{thm:main} begins with the observation that there exist eigenfunctions of \eqref{eq:bvp1}--\eqref{eq:bvp2} on the unit disk $\mathbb D$ for which both the Helmholtz and the SP component vanish to high order at the origin. We then engineer suitably small perturbations of $\mathbb D$ such that the Helmholtz component $v$ of the perturbed eigenfunction $u$ still vanishes to high order at the origin but where $w(0)$ is nonzero. Thus, there exists a neighborhood of the origin where $u$ is approximately equal to $w$, not to~$v$. The theorem then follows by estimating the size of the disk centered at the origin where $w$, and thus $u$, does not change sign.

In the remainder of this section, we sketch the construction of this perturbed domain and demonstrate heuristically that $u > 0$ on a disk of radius approaching a fixed radius (around $0.44$) as the frequency of the original, unperturbed eigenfunction becomes very large.

\subsubsection{Domain perturbations}
Let us consider a ``spacetime'' domain $\mathcal D \subseteq \mathbb R^2 \times (-\varepsilon,\varepsilon)$. We write $D_t = \mathcal D \cap (\mathbb R^2 \times \{t\})$ for the time slices, which we assume to be smooth and bounded. Assume there is a smooth vector field $X$, defined on a neighborhood of $\mathcal D$, possibly time-dependent, but parallel to the time slices $\{t\} \times \mathbb R^2$, such that $\partial_t + X$ is tangent to $\partial \mathcal D$. Let $u_t$ be a smooth family of solutions of \eqref{eq:bvp1}-\eqref{eq:bvp2} on $D_t$ with eigenvalue $\lambda_t$, normalized so that $\int_{D_t} u_t^2 = 1$ for all $t \in (-\varepsilon,\varepsilon)$. In the following, we suppress the subscript $t$ and write $\dot u = \frac{\partial u_t}{\partial t}, \ddot u = \frac{\partial^2 u_t}{\partial t^2}$.

\subsubsection{Shape derivative of the eigenvalue}
We first establish three simple claims, which we then use to derive a Hadamard-type formula for the derivative of the eigenvalue $\lambda^2$. In the following $n$ denotes the outward-pointing unit normal of $\partial D$.
\begin{enumerate}[label=(\arabic*)]
    \item \label{eq:orthogonality} $\int_{D} \dot u u = 0$. 
    \item \label{eq:derivativeOfZerothOrderBC} $\dot u = 0$ on $\partial D$. 
    \item \label{eq:derivativeOfFirstOrderBC} $\partial_n \dot u + (X \cdot n) \Delta u = 0$ on $\partial D$. 
    \item \label{eq:derivativeOfEigenvalue} $2 \lambda \dot \lambda = -\int_{\partial D} (X \cdot n) |\Delta u|^2$. 
\end{enumerate}

Claim \ref{eq:orthogonality} follows from $(u,u)_{L^2(D)} = 1$ and $u|_{\partial D} = 0$. Applying $\partial_t + X$ to $u$ and recalling that $u|_{\partial \mathcal D} = 0$, $\partial_n u|_{\partial \mathcal D} = 0$ and $\partial_t + X \parallel \partial \mathcal D$ yields \ref{eq:derivativeOfZerothOrderBC}. Applying $\partial_t + X$ to $\partial_n u$, we obtain $\partial_n \dot u + (X \cdot n) \partial_n^2 u = 0$ on $\partial D$. In curvilinear coordinates adapted to $\partial D$, it is standard that the Euclidean Laplacian decomposes as $\Delta = \Delta_{\partial D} + \partial_n^2 + H \partial_n$ (where $H$ denotes the mean curvature of $\partial D$), yielding \ref{eq:derivativeOfFirstOrderBC} since $u = \partial_n u = 0$ on $\partial D$.
Finally, the expressions \ref{eq:orthogonality}--\ref{eq:derivativeOfFirstOrderBC} allow us to calculate the shape derivative of the eigenvalue, starting from the observation that $\lambda^2 = \int_D |\Delta u|^2$, which follows from an integration by parts:
\begin{align*}
    \begin{split}
    2 \dot \lambda \lambda &= \int_{\partial D} (X \cdot n) |\Delta u|^2 + 2 \int_D \Delta \dot u \Delta u = \int_{\partial D} (X \cdot n) |\Delta u|^2 + 2 \int_{\partial D} \partial_n \dot u \Delta u - 2 \int_D \nabla \dot u \nabla \Delta u \\
    &= - \int_{\partial D} (X \cdot n) |\Delta u|^2 -  \int_D \nabla \dot u \nabla \Delta u = - \int_{\partial D} (X \cdot n) |\Delta u|^2 - \int_{\partial D} \dot u \partial_n \Delta u + \int_D \dot u \Delta^2 u \\
    &= - \int_{\partial D} (X \cdot n) |\Delta u|^2
    \end{split}
\end{align*}

\subsubsection{Shape derivatives of the Helmholtz and SP components}

Recall the Helmholtz and SP components of $u$, given by $v = \frac{1}{2\lambda}(\Delta u - \lambda u)$, $w = \frac{1}{2\lambda}(\Delta u + \lambda u)$.
Using Claims~\ref{eq:orthogonality}--\ref{eq:derivativeOfEigenvalue} above, it is easy to see that  their derivatives, $\dot v$ and $\dot w$, satisfy the following system of equations:
\begin{enumerate}[label=(\arabic*)]
    \setcounter{enumi}{4}
    \item \label{eq:componentVEquation} $(\Delta + \lambda) \dot v + \dot \lambda v = 0$.
    \item \label{eq:componentWEquation} $(\Delta - \lambda) \dot w - \dot \lambda w = 0$.
    \item \label{eq:componentsZerothOrder} $\dot v - \dot w = 0$ on $\partial D$.
    \item \label{eq:componentsFirstOrder} $\partial_n \dot v - \partial_n \dot w = (X \cdot n) \Delta u = 2 \lambda (X \cdot n) w$ on $\partial D$.
\end{enumerate}

Most important for our purposes is the expression for the second variation of~$u$ on the boundary. Indeed, applying $\partial_t + X$ to \ref{eq:derivativeOfZerothOrderBC} yields $\ddot u + (X \cdot n) \partial_n \dot u = 0$. By plugging in \ref{eq:derivativeOfFirstOrderBC}, we then obtain
\begin{enumerate}[label=(\arabic*)]
    \setcounter{enumi}{8}
    \item \label{eq:secondDerivative} $\ddot u = (X \cdot n)^2 \Delta u$ on $\partial D$.    
\end{enumerate}
This formula underpins the proof of \Cref{thm:main}: it shows that $w$ and $v$, which agree up to first order on $\partial D_0$, may differ in second order. We will use this fact to construct a perturbation of the unit disk on which the first few components of the Fourier-Bessel expansion of $v$ vanish, but the radial component of the expansion of $w$ into modified Bessel functions does not. This component, which does not change sign, then drives the behavior of $u$ near zero.

\subsubsection{Construction of a suitable deformation}

We begin by considering the unit disk $\mathbb D$ in $\mathbb R^2$. Each eigenspace of the clamped plate problem is an $SO(2)$-representation, via the standard action by rotation on $\mathbb D$. Fix $N > 0$. We consider the first eigenspace of the clamped plate problem with angular momentum $N$, and denote the corresponding eigenvalue by $\lambda^2$. The eigenspace is two-dimensional and spanned by
\begin{align}
    \label{eq:eigenfunctions}
    u_+(r,\theta) = \frac1{\sqrt{\pi}} \cos(N \theta)(J_N(\sqrt{\lambda})^{-1} J_N(\sqrt{\lambda} r) - I_N(\sqrt{\lambda})^{-1} I_N(\sqrt{\lambda} r)), \\
    u_-(r,\theta) = \frac1{\sqrt{\pi}} \sin(N \theta)(J_N(\sqrt{\lambda})^{-1} J_N(\sqrt{\lambda} r) - I_N(\sqrt{\lambda})^{-1} I_N(\sqrt{\lambda} r)),
\end{align}
where $J_N$ is the $N^{th}$ Bessel function and $I_N$ is the $N^{th}$ modified Bessel function. The constant $\frac{1}{\sqrt{\pi}}$ in front is chosen so that $u_+$ and $u_-$ are $L^2$-normalized:

\begin{lemma}
    \label{lem:normalization}
    The $L^2$-norm of the clamped plate eigenfunctions $u_+,u_-$ on the unit disk is~$1$.
\end{lemma}

\begin{proof}
Let us show that if $u(r,\theta) = c_{N,\lambda} \cos(N \theta)(J_N(\sqrt{\lambda})^{-1} J_N(\sqrt{\lambda} r) - I_N(\sqrt{\lambda})^{-1} I_N(\sqrt{\lambda} r))$ is $L^2$-normalized on the unit disk, then $c_{N,\lambda} = \pm \frac{1}{\sqrt{\pi}}$. The same argument applies if one replaces the cosine by a sine.

    This is easiest to establish by considering the behaviour of the eigenvalue $\lambda$ along a deformation induced by the scaling vector field $r \,\partial_r$. The image of the unit disk along the time-$t$ flow generated by $r\,\partial_r$ is a disk of radius $e^{t}$. The scaling behavior of $\Delta^2$ implies $\lambda_t^2 = e^{-4t} \lambda^2$. Comparing this with the expression obtained in \ref{eq:derivativeOfEigenvalue} (which assumes that the eigenfunction $u$ is $L^2$-normalized), we obtain
    \begin{align*} 
        -4 \lambda^2 = \frac{d}{dt}\bigg|_{t=0} (\lambda_t^2) = -4 \lambda^2 c_{N,\lambda}^2 \int_{\partial \mathbb D} \cos(N\theta)^2,
    \end{align*}
    implying that $c_{N,\lambda}^2 = \frac{1}{{\pi}}$ as claimed.
\end{proof}

The rotation around the origin by an angle of $\frac{2\pi}{N}$ and the reflection around the $x$-axis together generate the action of the $N^{th}$ dihedral group $D_N$ on the unit disk. We will perturb the unit disk such that we retain the action of the dihedral group $D_N$ on the domain. The space of $D_N$-invariant functions is preserved by $\Delta^2$. It can be shown that, with $N$ and $\lambda^2$ chosen as above, the eigenvalue $\lambda^2$ is simple among the $D_N$-invariant spectrum of $\mathbb D$ for infinitely many $N$ (see \Cref{lem:nondegenerate_eigenvalues}). The corresponding one-dimensional eigenspace is then spanned by $u_+$. Thus, we may apply the regular perturbation theory of a simple eigenvalue to analyze the behavior of $u_+$ (forthwith denoted by $u$) and $\lambda^2$ under a deformation of $\mathbb D$ which preserves the $D_N$-symmetry.

There is a codimension two manifold of the space of $D_N$-invariant deformations on which $v_t(0) = 0$ and $\lambda_t = \lambda_0$. This is a consequence of the implicit function theorem and the nondegeneracy and linear independence of the functionals $\delta \lambda$ and $\delta v(0)$, To see this, we first note that $\delta_X \lambda =-2\lambda \neq 0$, where $X=r\,\partial_r$ is the scaling vector field. Next, consider a deformation along a direction $Y \in \ker \delta\lambda$. In this case, $\dot \lambda = 0$, so \ref{eq:componentVEquation}--\ref{eq:componentsFirstOrder} take the shape
\begin{align*}
    &(\Delta + \lambda) \dot v = 0 \\
    &(\Delta - \lambda) \dot w = 0 \\
    &\dot v = \dot w \text{ on } \partial D \\
    & \partial_n \dot v - \partial_n \dot w = 2 \lambda (Y \cdot n) w \text{ on } \partial D.
\end{align*}

This allows us to compute $\dot v(0)$. (All calculations in the remainder of this section will be performed at $t=0$, which we will therefore suppress in the notation.) Indeed, consider the Fourier--Bessel expansion of $\dot u$,
\begin{align*}
    \dot u(r,\theta) = \sum_{n=0}^\infty \cos(n \theta) \left(\dot A_n J_n(\sqrt{\lambda})^{-1} J_n(\sqrt{\lambda} r) + \dot B_n I_n(\sqrt{\lambda})^{-1} I_n(\sqrt{\lambda} r) \right)
\end{align*}
Since $\dot v = \dot w$ on $\partial \mathbb D$, $\dot A_n + \dot B_n = 0$.
Taking the $0^{th}$ Fourier coefficient of $(\partial_n \dot v - \partial_n \dot w)|_{\partial \mathbb D} = 2 \lambda (Y \cdot n) v$ yields
\begin{align*}
    \sqrt{\lambda}\, \dot A_0 \left(J_0(\sqrt{\lambda})^{-1} J_0'(\sqrt{\lambda}) - I_0(\sqrt{\lambda})^{-1} I_0'(\sqrt{\lambda})\right) = \frac{\lambda}{\sqrt{\pi^3}} \int_{\partial \mathbb D} \cos(N \theta) (Y \cdot n).
\end{align*}
The expression $J_0(\sqrt{\lambda})^{-1} J_0'(\sqrt{\lambda}) - I_0(\sqrt{\lambda})^{-1} I_0'(\sqrt{\lambda})$ is nonzero if and only if $\lambda^2$ is not also a radial eigenvalue of the clamped plate problem. This can be arranged easily, as we will prove in \Cref{sec:disk}. 

Under this assumption,
\begin{align}
    \label{eq:derivative_of_v0}
    \dot v(0) = \frac{\sqrt{\lambda}}{\sqrt{\pi^3}} J_0(\sqrt{\lambda})^{-1} \left(J_0(\sqrt{\lambda})^{-1} J_0'(\sqrt{\lambda}) - I_0(\sqrt{\lambda})^{-1} I_0'(\sqrt{\lambda})\right)^{-1} \int_{\partial \mathbb D} \cos(N \theta) (Y \cdot n)
\end{align}
Thus, $\delta_Y v(0) \neq 0$ if and only if $\int_{\partial \mathbb D} \cos(N \theta) (Y \cdot n) \neq 0$; in particular, $\delta_{(\cdot)} v(0)$ does not vanish.
We conclude that, since for the scaling vector field $X$ we have $\delta_X \lambda \neq 0$ and $\delta_X v(0) = 0$, the two functionals $\delta \lambda$ and $\delta v(0)$ are linearly independent. Let us also record the formula for the variation of the eigenvalue corresponding to $u_+$ here: Since $\Delta u|_{\partial D} = - \frac{2 \lambda}{\sqrt{\pi}} \cos(N \theta)$,
\begin{align}
    \label{eq:derivative_of_lambdasquared}
    \frac{d}{dt}|_{t=0} (\lambda_t^2) = - \int_{\partial D} (X \cdot n) |\Delta u_+|^2 = - \frac{2\lambda^2}{\pi} \int_{\partial D} (X \cdot n) (1 + \cos(2 N \theta))
\end{align}

The implicit function theorem then furnishes a codimension 2 manifold $\mathcal M$ of $D_N$-symmetric domains, with $\mathbb D\in \mathcal M$, where $\lambda$ and $v(0) = 0$ are independent of the domain in~$\mathcal M$. Furthermore, it follows from the preceding calculations that 
\begin{align}
    T_{\mathbb D}\mathcal M &= \left\{X: \int_{\partial \mathbb D} (X \cdot n) |\Delta u|^2 = 0, \int_{\partial \mathbb D} (X \cdot n) \Delta u = 0\right\} \\ &= \left\{X: \int_{\partial \mathbb D} \left(1+ \cos(2N \theta)\right) (X \cdot n)(\theta) d\theta = 0, \int_{\partial \mathbb D} \cos(N\theta) (X \cdot n)(\theta) d\theta = 0 \right\}
\end{align}
Of course, only the normal component of the vector field~$X$ on the boundary of the domain is relevant for the characterization of the tangent space, as domain deformations are determined  to first order precisely by this  component.

\subsubsection{Consequences for the nodal set: Overview of the remaining proof}
Let us consider a smooth curve $(D_t)_{t \in (-\varepsilon,\varepsilon)}$ of domains in the manifold $\mathcal M$ just constructed. Then $\Delta v_t + \lambda v_t = 0$ for all $t$ (since $\lambda_t = \lambda$ is constant), so $\Delta \ddot v + \lambda \ddot v = 0$ as well, and analogously $\Delta \ddot w + \lambda \ddot w = 0$. This allows us to take the Fourier--Bessel expansions of $v$ and $w$. All components in these  expansions with angular frequency not divisible by $N$ vanish identically because $u_t$ is $D_N$-invariant by construction, so one has
\begin{align*}
    \ddot u(r,\theta) = \ddot w(r,\theta) - \ddot v(r,\theta) = \sum_{k=0}^\infty\left[\ddot A_{kN} J_{kN}(\sqrt{\lambda} r) + \ddot B_{kN} I_{kN}(\sqrt{\lambda} r) \right]\cos(kN\theta).
\end{align*}
Because the curve of deformations under consideration is contained in $\mathcal M$, $\ddot v(0)=\ddot A_0 = 0$. 

Let $X$ denote the deformation velocity at $t=0$. Recalling that $\ddot u = (X \cdot n)^2 \Delta u$ on $\partial \mathbb D$ by \ref{eq:secondDerivative}, we infer
\begin{align*}
    \ddot B_0 = I_0(\sqrt{\lambda})^{-1}\frac{1}{2\pi}\int_{\partial \mathbb D}\ddot u= I_0(\sqrt{\lambda})^{-1} \frac{1}{2\pi}\int_{\partial \mathbb D} (X \cdot n)^2 \Delta u.
\end{align*}

Write $\alpha = \frac{1}{2\pi} \int_{\partial \mathbb D} (X \cdot n)^2 \Delta u$. We claim that if we take
\begin{align}
    \label{eq:conditionOnX3}
    X \cdot n = \cos(3N\theta) + \cos(2N\theta) - \frac12,
\end{align}
then  $X \in T_{\mathbb D}\mathcal M$ and $\alpha \neq 0$. First, $\int_{\mathbb D} (X \cdot n) \cos(N\theta) d\theta = \int_{\mathbb D} (X \cdot n) (\cos(2N\theta) + 1) d\theta = 0$, so indeed $X \in T_{\mathbb D}\mathcal M$ by our characterization of the tangent space. On the other hand, $\int_{\partial \mathbb D} (X\cdot n)^2 \cos(N \theta) d\theta = \pi$ and $\Delta u|_{\partial \mathbb D} = - \frac{2\lambda}{\sqrt{\pi}} \cos(N \theta)$, hence
\begin{align*}
    \alpha = \frac{1}{2\pi} \int_{\partial \mathbb D} (X \cdot n)^2 \left( - \frac{2 \lambda }{\sqrt{\pi}} \cos(N\theta)\right) d\theta = -\frac{\lambda}{\sqrt{\pi}}.
\end{align*}
It follows that
\begin{align}
    \label{eq:secondDerivativeOfW}
    \ddot w(0) = \ddot B_0 I_{0}(0) = I_0(\sqrt{\lambda})^{-1} \alpha = - \frac{\lambda}{\sqrt{\pi}} I_0(\sqrt{\lambda})^{-1}.
\end{align}

We henceforth consider a curve of domains in $\mathcal M$ with initial deformation velocity given by ~\eqref{eq:conditionOnX3}. For small $t$ and $r$, we thus find that
\begin{align*}
    v_t(r,\theta) &= \left(\frac{1}{\sqrt{\pi}} + \mathcal O(t)\right) \cos(N \theta) J_N(\sqrt{\lambda})^{-1} J_N(\sqrt{\lambda} r) + \mathcal O_\lambda(r^{2N}), \\
    w_t(r,\theta) &= \left(- \frac{\lambda}{\sqrt{\pi}} \frac{t^2}{2} + \mathcal O(t^3)\right) I_0^{-1}(\sqrt\lambda) I_0(\sqrt{\lambda} r) + \mathcal O_\lambda(r^{N}e^{\sqrt{\lambda} (r-1)}).
\end{align*}
This is the setting in which we can close to argument.
Let us now sketch how the proof goes, modulo a number of estimates that we will establish in later sections.

The first term in Debye's asymptotic expansion for $J_N(N(\cdot))$ yields that, given $r \in (0,\frac{N}{\sqrt{\lambda}})$,
\begin{align*}
    \log|J_N(\sqrt{\lambda})^{-1} J_N(\sqrt{\lambda} r)| \approx N \varrho\left(\frac{\sqrt{\lambda}}{N} r\right),
\end{align*}
with $\varrho(s) = \log(s) + \sqrt{1+s^2} - \log(1 + \sqrt{1+s^2})$. 
On the other hand, $I_0(\sqrt{\lambda})^{-1} I_0(\sqrt{\lambda} r) \geq \exp(-\sqrt{\lambda}(1-r))$.
\Cref{lem:bounds} below provides sharp uniform bounds that provide the proper framework to use these approximations effectively. 

Next, \Cref{lem:remainder_bounds} below shows that the error terms arising from higher modes in the series expansions of $v_t$ and $w_t$ are negligible in comparison to the leading terms, in a precise sense. Combining these facts, we find that $|v_t(r,\theta)| \leq |w_t(r,\theta)|$ if
\begin{align*}
    &N \varrho\left(\frac{\sqrt{\lambda}}{N} r\right) + c_0 + c_1 \log(\lambda) \leq - \sqrt{\lambda} (1-r) + 2 \log t,
\end{align*}
where $c_0, c_1 > 0$ are constants independent of $\lambda$ which will be determined later in such a way that $c_0 + c_1 \log(\lambda)$ collects all error terms which depend at most polynomially on $\lambda$. Upon rearranging, we thus conclude that $|v_t(r,\theta)| \leq |w_t(r,\theta)|$ if
\begin{align*}
    &\varrho\left(\frac{\sqrt{\lambda}}{N} r\right) - \frac{\sqrt{\lambda}}{N} r + \frac{\sqrt{\lambda}}{N} \leq \frac{2 \log t - c_0 - c_1 \log(\lambda)}{N},
\end{align*}

Already for $t \sim \exp(- \sqrt{\lambda})$, a circle of a fixed radius (asymptotically independent of $\lambda$) will therefore be free of nodal lines, since $\sqrt{\lambda} \approx N$. To improve the size of this disk, we will refine the argument and construct larger deformations with $t $ of order $ \mathcal O(\lambda^{-\beta})$, for which  the asymptotic calculation above is still valid. In this setting, a lower bound for this radius where the eigenfunction does not vanish is given by the (unique) solution $\xi_N\in (0,1)$ to the equation
\begin{align*}
    \varrho\left(\frac{\sqrt{\lambda}}{N} \xi_N\right) - \frac{\sqrt{\lambda}}{N} \xi_N + \frac{\sqrt{\lambda}}{N} \leq \frac{- c_0 - (2 \beta + c_1) \log(\lambda)}{N},
\end{align*}
The right hand side of this expression converges fairly quickly to $0$, so $\xi_N$ is well-approximated by $\frac{N}{\sqrt{\lambda}} \zeta_N$, where $\zeta_N$ is the unique zero of the function $\varrho(x) - x + \frac{\sqrt{\lambda}}{N}$ in $(0,1)$. Since we consider the lowest mode with angular parameter $N$, we have that $\frac{\sqrt{\lambda}}{N} \to 1$ as $N \to \infty$, with uniform bounds in this convergence provided in \Cref{lem:nondegenerate_eigenvalues}. This will enable us to show that $\xi_N$ approaches the unique zero $\zeta_\infty =0.44\dots$ of the function $\varrho(x) - x + 1$. This convergence is significantly slower, as $|\frac{\sqrt{\lambda}}{N} - 1| \simeq N^{-\frac23}$. 

Thus the argument will ensure the existence of clamped plate eigenfunctions on $D_N$-symmetric domains close to a disk which have a region without nodal lines that contains a circle of radius approaching $\zeta_\infty =0.44\dots$ as $N \to \infty$. To rigorously establish this fact, it remains to bound the contributions of
\begin{itemize}
    \item higher derivatives of $v_t$ and $w_t$ and
    \item higher modes in the Fourier--Bessel expansion of $v_t$ and $w_t$
\end{itemize}
to these inequalities. In addition to this, we must show that there exists a sequence of $N$, going to infinity, such that the lowest eigenvalue $\lambda_N$ of the clamped plate problem with angular momentum $N$ stays away from the radial spectrum (i.e., the set of eigenvalues to which there exists an eigenfunction invariant under rotation), which will be the content of~\Cref{lem:nondegenerate_eigenvalues}. This is necessary for two reasons. First, our expression for the derivative $\delta_Y v(0)$ contains a term which is not well-defined when $\lambda_N$ is contained in the radial spectrum. Secondly, in any spectral problem, bounds for the derivatives of eigenfunctions with respect to variation of a parameter will depend on the distance to the remaining spectrum. 

To facilitate the reading of the following sections, which carry out these tasks, Figure~\ref{fig:placeholder} presents a flow chart illustrating the logical connections among the various technical results that need to be established. \Cref{lem:concreteEstimates}, \Cref{lem:oneParameterFamily} and the conclusion of the proof of Theorem 1 form the main body of the remaining argument. They rest on derivative estimates for eigenfunctions in abstract eigenvalue problems (\Cref{lem:abstractEstimates}), are supplemented with explicit estimates on the derivatives of the operator families provided, and their hypotheses are always informed by \Cref{lem:nondegenerate_eigenvalues}, which formally enters the proof only at the end in the form of a lower bound on a spectral gap.

\begin{figure}
    \centering
    \begin{tikzpicture}[node distance=2.5cm]
    \scriptsize

        \node (Lem9) [startstop, text width=4cm] {\textbf{\Cref{lem:concreteEstimates}.} Derivative bounds on clamped plate eigenfunctions along multi-parameter deformations.};
        \node (Lem11) [startstop, text width=4cm, below of = Lem9] {\textbf{\Cref{lem:oneParameterFamily}.} One-parameter deformation with derivative bounds from implicit function theorem.};
        \node (Lem12) [startstop, below of=Lem11, text width = 4cm] {\textbf{Proof of \Cref{thm:main}.} Error estimates from series remainders and higher derivatives.};

        \node (Lem4) [io, text width=4cm, above of = Lem9, xshift=-2.5cm] {\textbf{\Cref{lem:nondegenerate_eigenvalues}.} Lower bound for spectral gap.};

        \node (Lem5) [process, text width=4cm, left of = Lem9, xshift = -3cm] {\textbf{\Cref{lem:bound_on_L2norms}.} Bounds on Bessel integrals.};
        \node (Lem3) [process, text width=4cm, below of = Lem5] {\textbf{\Cref{lem:bounds}.} Bounds on Bessel functions.};
        \node (Lem6) [process, text width=4cm, below of = Lem3] {\textbf{\Cref{lem:remainder_bounds}.} Bounds on Bessel series remainders.};
        
        \node (Lem7) [process, text width=4cm, right of = Lem9, xshift = 3cm] {\textbf{\Cref{lem:abstractEstimates}.} Derivative estimates on eigenfunctions of parametrized spectral problems.};
        \node (Lem10) [process, text width=4cm, below of = Lem7] {\textbf{\Cref{lem:explicitNorms}.} $C^k$-norms of deformation vector fields.};
        \node (Lem8) [process, text width=4cm, above of = Lem7] {\textbf{\Cref{lem:derivativeEstimatesForLaplacian}.} Derivative estimates for the Laplacian in terms of the metric.};

        \draw [arrow] (Lem8) -- (2.5,2.5) -- (2.5,0);
        \draw [arrow] (Lem7) -- (Lem9);
        \draw [arrow] (Lem9) -- (Lem11);
        \draw [arrow] (Lem10) -- (2.5,-2.5) -- (2.5,-1.1) -- (0,-1.1);
        \draw [arrow] (Lem11) -- (Lem12);
        \draw [arrow] (Lem3) -- (Lem6);
        \draw [arrow] (Lem5) -- (Lem9);
        \draw [arrow] (Lem3) -- (Lem5);
        \draw [arrow] (Lem6) -- (Lem12);
        \draw [arrow] (Lem4) -- (-2.5,-4.8) -- (-2.12,-4.8);
    \end{tikzpicture}
    \caption{Structure of the remaining proof}
    \label{fig:placeholder}
\end{figure}
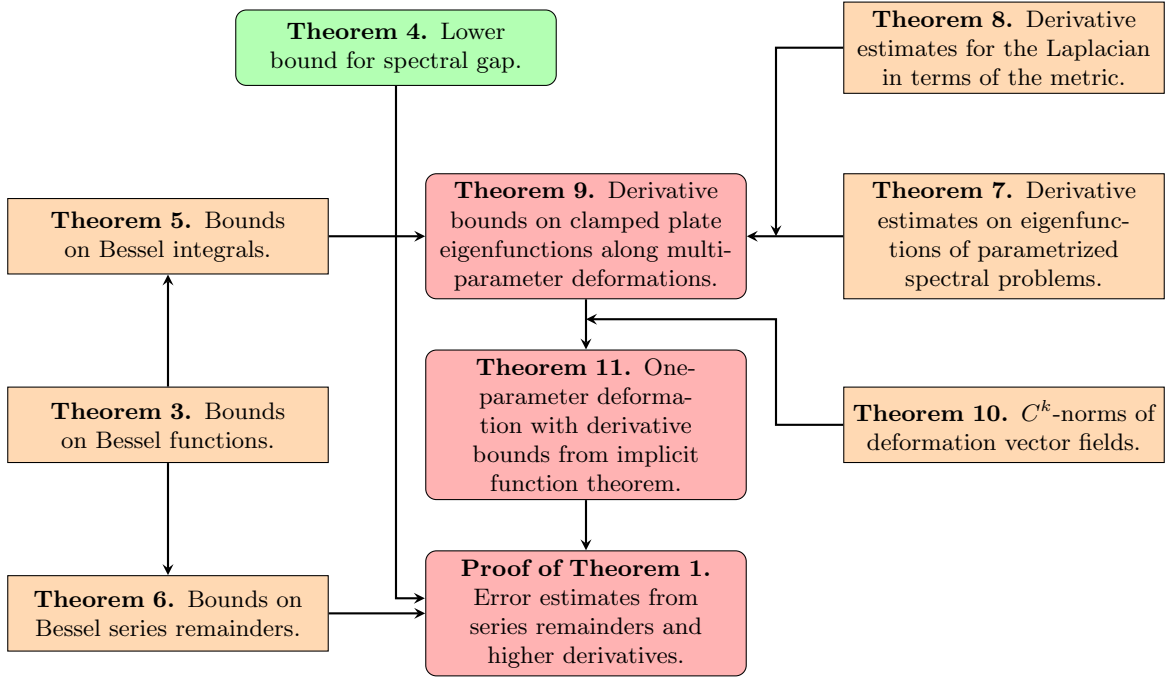

\section{Uniform bounds on Bessel functions}
\label{sec:disk}

Throughout the proof, we need uniform upper and lower bounds on Bessel and modified Bessel functions, with small \emph{relative} error in all the relevant argument regimes. The following lemma provides such bounds. Although these are of course closely related to asymptotic expressions known since the nineteenth century, several of the bounds below are not readily found in the literature, and it is convenient to collect these bounds and their simple proofs here.

\begin{lemma}
    \label{lem:bounds}
    For the Bessel functions $J_n$ and modified Bessel functions $I_n$, the following bounds hold.
    \begin{enumerate}
        \item For any $x \in (0,1)$, $I_0(\sqrt{\lambda})^{-1} I_0(\sqrt{\lambda} x) > e^{- \sqrt{\lambda}(1-x)}$.
        \item For any integer $n \geq 1$ and $x \in (0,1)$,
        \begin{align*}
            0 \leq n \left(\sqrt{1-x^2} + \log\left(\frac{x}{1 + \sqrt{1-x^2}}\right)\right) - \log J_n(n x) \leq \frac{1}{2}\log(2 \pi n) + \frac{1}{12n}
        \end{align*}
        \item For any integer $n \geq 1$ and $x \in (0,\infty)$,
        \begin{align*}
            &\log I_n(n x) - n \left(\sqrt{1+x^2} + \log\left(\frac{x}{1 + \sqrt{1+x^2}}\right)\right) + \frac14 \log\left(1 + \frac1n + x^2\right) + \frac{1}{2}\log(2 \pi n)
        \end{align*}
        is contained in the interval $(0,\frac1{6n})$.
    \end{enumerate}
\end{lemma}
\begin{proof}
    \begin{enumerate}
        \item On $(0,\infty)$, $w(x) := I_0(x)$ satisfies the ordinary differential equation
        \begin{align*}
            w'' + \frac1x w' - w = 0.
        \end{align*}
        Thus, $\frac{d}{dx}[(w')^2-w^2] = - \frac2x (w')^2 < 0$. Since $w'(0)^2 - w(0)^2 = -1$, $w'(x) < w(x)$ on $(0,\infty)$. Hence, $w$ is a subsolution for the ordinary differential equation $y' = y$, which means that $e^{t-s} w(s) > w(t)$ for all $0 < s < t < \infty$. Plugging in $t = \sqrt{\lambda}$, $s = x \sqrt{\lambda}$ yields $I_0(\sqrt{\lambda})^{-1} I_0(\sqrt{\lambda} x) > e^{- \sqrt{\lambda}(1-x)}$ as claimed.
        \item Since $f(x) = J_n(x)$ solves the Bessel equation 
        \begin{align*}
            f''(x) + \frac{1}{x} f'(x) + \left(1-\frac{n^2}{x^2}\right) f(x) = 0,
        \end{align*}
        the derivative $q(y)$ of $\frac{1}{n} \log J_n(n y)$ solves the Riccati equation
        \begin{align}
            \label{eq:Riccati}
            \frac{1}{n} q'(y) + q(y)^2 + \frac{1}{ny} q(y) + 1 - \frac{1}{y^2} = 0.
        \end{align}
        An explicit calculation shows $p_\sigma(y) = \sigma \frac{\sqrt{1-y^2}}{y}$ is a subsolution for \eqref{eq:Riccati} for all $\sigma \in (0,1]$: Since $\frac{p_\sigma'}{p_\sigma} = - \frac{y}{1-y^2} - \frac1y$, we have $p_\sigma' + \frac1y p_\sigma < 0$, and $p_\sigma^2 + 1-\frac1{y^2} \leq 0$. It is well known that $J_n(x) = (\frac12 x)^n/\Gamma(n+1)$ as $x \to 0$ (see \cite[9.1.61]{AbramowitzStegun1964}). Thus, $q(y) = \frac 1y + \mathcal O(1)$ as $y \to 0$. In particular, $p_\sigma(y) < q(y)$ for any fixed $\sigma \in (0,1)$ and $y$ small enough. Since $p_\sigma$ is a subsolution, $p_\sigma(y) < q(y)$ for all $y \in (0,1)$. Taking $\sigma \to 1$, we find that $p(y) := p_1(y) < q(y)$ for all $y \in (0,1)$ (the inequality is strict since $p$ is a subsolution as well). Therefore, the antiderivative $\int (p(y)-q(y))dy$, explicitly given by
        \begin{align*}
            E(y) := \left( \log(y) + \sqrt{1-y^2} - \log(1+\sqrt{1-y^2})\right) - \frac1n \log J_n(ny),
        \end{align*}
        is strictly decreasing. Since $J_n(x) \sim (\frac12 x)^n/\Gamma(n+1)$ as $x \to 0$ and $J_n(n) <2^{\frac13}/(3^{\frac23}\Gamma(\frac23) n^{\frac13})$ (\cite[9.1.61]{AbramowitzStegun1964}, originally due to Cauchy),
        \begin{align*}
            \lim_{y \to 0} E(y) &= 1 - \log n + \frac1n \log \Gamma(n+1) < \frac1n \left( \frac12 \log(2 \pi n) + \frac{1}{12n} \right), \\
            E(1) &= - \frac1n \log J_n(n) > \frac1n \left( \log\left(\frac{3^{\frac23}\Gamma(\frac23)}{2^{\frac13}}\right) + \frac13 \log n \right) > 0.
        \end{align*}
        This yields the claimed bounds on $n p(x) - \log J_n(nx)$.
        \item Let $q(y)$ now denote the derivative of $\frac1n \log I_n(y)$. The function $q$ solves the Riccati equation
        \begin{align}\label{eq_lastric}
            \frac{1}{n} q'(y) + q(y)^2 + \frac{1}{ny} q(y) - 1 - \frac{1}{y^2} = 0,
        \end{align}
        which is derived from the modified Bessel differential equation. For $c \in [0,1]$, consider
        \begin{align*}
            p_c(y) := \frac{\sqrt{1+y^2}}{y} -  \frac{y}{2n(1+c + y^2)}.
        \end{align*}
        A tedious but straightforward calculation shows  that
        \begin{align*}
            &\frac{1}{n} p'(y) + p(y)^2 + \frac{1}{ny} p(y) - 1 - \frac{1}{y^2} \\
            = \ & \frac{c^2n + (\frac{y^2}{4} - 1)\sqrt{1+y^2} + c(n+ny^2 - \sqrt{1+y^2})}{n^2\sqrt{1+y^2}(1+c+x^2)^2}
        \end{align*}
        The denominator is positive. Writing $t = \sqrt{1+y^2}$, the numerator becomes
        \begin{align*}
            c^2 n + (\tfrac14t^2 - \tfrac54)t + cn(t^2 - \tfrac1nt).
        \end{align*}
        If $c = \frac1n$, the above factorizes in $(t-1)(\frac14 t^2 + \frac54 t - \frac1n)$, which is positive for $t > 1$. Since $n \geq 1$, the numerator can be seen to increase with $c$. Thus, $p_c$ is a supersolution to the Riccati equation~\eqref{eq_lastric} for $c \in [\frac1n,1]$. From the standard power series expansion of $I_n(z)$ around $0$, we infer that
        \begin{align*}
            q(y) = \frac{1}{y} + \frac{ny}{2n+2}  + \mathcal O (|y|^3).
        \end{align*}
        On the other hand, $p_c(y) = \frac{1}{y} + \frac y2(1 - \frac{1}{n(1+c)}) + \mathcal O(|y|^3) > \frac{1}{y} + \frac{n}{n+1} \frac{y}{2} + \mathcal O (|y|^3)$ if $c \geq \frac1n$. Thus, $p_c(y) > q(y)$ for small $y$ if $c \in (\frac1n,1]$. Since $p_c$ is a supersolution, $p_c(y) \geq q(y)$ for all $y > 0$. Taking $c \to \frac1n$, we obtain $p_{\frac1n}(y) \geq q(y)$ for all $y > 0$, but since $p_{\frac1n}$ is a supersolution as well, $p_{\frac1n}(y) > q(y)$ on $(0,\infty)$ follows. Thus, the antiderivative of $p_{\frac1n}(y) - q(y)$, which is explicitly given by
        \begin{align*}
            r(y) := \sqrt{1+y^2} + \log\frac{y}{1+\sqrt{1+y^2}} - \frac1{4n} \log\left(\frac{n+1}{n}+y^2\right) - \frac{1}{n}I_n(ny),
        \end{align*}
        is increasing on $(0,\infty)$. From the modified Bessel function's large-argument asymptotics $I_n(z) \sim e^z/\sqrt{2 \pi n}$ (\cite[9.7.1]{AbramowitzStegun1964}) we infer that $r(y) < \frac1{2n} \log(2 \pi n)$. The small argument asymptotics \cite[9.6.7]{AbramowitzStegun1964} $I_n(z) \sim \left( \frac{z}{2} \right)^n/n!$ imply 
        \begin{align*}
            r(y) > \lim_{t \to 0} r(t) = \frac1n \log n! - \log n + 1 - \frac1{4n} \log\left(1 + \frac1n\right).
        \end{align*}
        Robbins' error bound~\cite{Robbins55} $\frac{1}{12n+1} < \log n! - n(\log(n)-1) - \frac12\log(2 \pi n) < \frac{1}{12n}$ and the estimate $\log(1+\frac1n) < \frac{2}{2n+1}$ let us simplify this to 
        \begin{align*}
            r(y) > \frac1{2n}\log(2 \pi n) + \frac{1}{n(12n+1)} - \frac{1}{2n(2n+1))} > \frac1{2n}\log(2 \pi n) - \frac1{6n^2},
        \end{align*}
        yielding the bound claimed in the lemma's statement. \qedhere
    \end{enumerate}
\end{proof}

Next, we prove that there exist infinitely many ``highly nondegenerate'' eigenvalues of the clamped plate problem in dihedral symmetry:

\begin{lemma}
    \label{lem:nondegenerate_eigenvalues}
    For any $N \in \mathbb N$, consider the function $W_N(x) = J_N(x)^{-1} J_N'(x) - I_N(x)^{-1} I_N'(x)$.
    There exist infinitely many $N \in \mathbb N$, $N \geq 100$, such that the first positive zero $\xi_{N,1}$ of $W_N$ is at a distance of at least $1$ from any zero of $W_0$, $- 6 \leq W_0(\xi_{N,1}) \leq - 1$, $\sqrt{\frac{2}{\pi}} \xi_{N,1}^{-\frac12} \geq J_0(\xi_{N,1}) \geq \frac{1}{10} \xi_{N,1}^{-\frac12}$ and $N + N^{1/3} < \xi_{N,1} < N + 3 N^{1/3}$. Finally, the corresponding clamped plate eigenvalue $\lambda^2 = \xi_{N,1}^4$ is at a distance of at least $4N^3$ from the remaining $D_N$-invariant spectrum of the clamped plate problem.
\end{lemma}

\begin{proof}
    To analyze the cross-ratio $W_0(x)$, it is helpful to write the Bessel function $J_0(x)$ in terms of the Bessel modulus and phase functions, defined so that $M_0(x)^2 = J_0(x)^2 + Y_0(x)^2$, where $Y_0$ is the corresponding Bessel function of second kind, and $J_0(x) = M_0(x)\cos(\theta_0(x))$. From the discussion in \cite[pg. 446]{Watson1966} it follows that $0 > \frac{M_0'(x)}{M_0(x)} > -\frac{1}{2x}$. The bounds on $M_0(x)^2$ provided in \cite[9.2.28]{AbramowitzStegun1964} yield $\frac{2}{\pi x}(1 - \frac{1}{8x^2}) < M_0(x)^2 < \frac{2}{\pi x}$. By \cite[9.2.21]{AbramowitzStegun1964}, $M_0(x)^2\theta_0'(x) = \frac{2}{\pi x}$, and hence $1 < \theta_0'(x) < (1 - \frac{1}{8x^2})^{-1}$. If $x > 10$, $(1 - \frac{1}{8x^2})^{-1} < 1 + \frac{1.01}{8x^2}$. Together with the large-argument asymptotics of $J_0$, it follows that $x - \frac\pi 4 - \frac{1.01}{8x} < \theta_0(x) < x - \frac\pi 4$.
    
    One may show $1 - \frac{1}{x} < I_0'(x)/I_0(x) < 1$, e.g.\ as follows: $q(x) = I_0'(x)/I_0(x)$ solves the ordinary differential equation $q' + q^2 + \frac 1x q - 1 = 0$, to which $p_+(x) = 1$ is a supersolution on $(0,\infty)$ and $p_-(x) = 1-\frac1x$ is a subsolution on $(1,\infty)$ satisfying $p_-(1)=0  < q(1)<1=p_+(1)$, so $p_- < q < p_+$.

    It follows that the expression 
    \begin{align}
        \label{eq:W0expression}
        W_0(x) = \frac{J_0'(x)}{J_0(x)} - \frac{I_0'(x)}{I_0(x)} = \frac{M_0'(x)}{M_0(x)} - \tan(\theta_0(x)) \theta_0'(x) - \frac{I_0'(x)}{I_0(x)}
    \end{align}
    is well-approximated by $-\tan(x-\frac\pi4)-1$ and so its zeros $\xi_{0,k}$ are well approximated by $(k+\frac12)\pi$. Precisely, $W_0(x) = 0$ and $x \geq 100$ imply
    \begin{align*}
        |\tan(\theta_0(x)) + 1| = |\theta_0'(x)^{-1}|\cdot\left|\frac{I_0'(x)}{I_0(x)} - \frac{M_0'(x)}{M_0(x)} - \theta_0'(x)\right| < \frac1x + \frac{1.01}{x^2} < \frac{1.02}{x}.
    \end{align*}
    Since $\tan'(x) \geq 1$, it follows that $\mathrm{dist}(\theta_0(x),\pi \mathbb Z + \frac34 \pi) < \frac{1.02}{x}$ and hence $\mathrm{dist}(x,\pi \mathbb Z) < \frac{1.25}{x}$.

    Next, we obtain information on $\xi_{n,1}$. The recurrence formulae $J_n'(x) = - J_{n+1}(x) + \frac{n}{x} J_n(x)$ and $I_n'(x) = I_{n+1}(x) + \frac{n}{x} I_n(x)$ (\cite[9.1.27, 9.6.26]{AbramowitzStegun1964}) imply that
    \begin{align*}
        W_n(x) = - \frac{J_{n+1}(x)}{J_n(x)} - \frac{I_{n+1}(x)}{I_n(x)}.
    \end{align*}
    Let $j_{n,k}$ denote the $k^{th}$ positive zero of $J_n$. Then $\xi_{N,1} \in (j_{n,1},j_{n+1,1})$. After all, $W_n(x) < 0$ on $(0,j_{n,1})$, $W_n(x) \to + \infty$ as $x \to j_{n,1}$ from above, and $W_n(x) = - \frac{I_{n+1}(x)}{I_n(x)} < 0$ at $j_{n+1,1}$.

    Lang and Wong~\cite[(1.5)]{Lang1996} obtained the following highly accurate bounds on $j_{n,1}$, valid for $n \geq 10$:
    \begin{align}
        \label{eq:LangWongBound}
        j_{n,1} = n + \frac{|a_1|}{2^{\frac13}} n^{\frac 13} + \frac{3}{20} a_1^2 \frac{2^{\frac13}}{n^{\frac13}} + \beta n^{-1},
    \end{align}
    where $- 0.060804 \leq \beta \leq - 0.000263$ and $a_1=-2.338\dots$ denotes the first negative zero of the Airy function. The length of the interval $(j_{n,1},j_{n+1,1})$ is thus bounded by $1 + |a_1|2^{-\frac13} 3^{-1} n^{-\frac23} + 0.07 n^{-1}$, which is less than $1.03$ if $n \geq 100$. In view of \eqref{eq:LangWongBound}, it is also clear that $j_{n,1}$ lies in $\pi \mathbb Z + (1.02,1.10)$ infinitely often, in which case $\mathrm{dist}(\xi_{n,1},\pi \mathbb Z) > 1.02$ and hence $|\xi_{n,1}-\xi_{0,k}| \geq 1$ for all $k \in \mathbb N$. It is also clear that 
    \begin{align*}
        n + \frac{|a_1|}{2^{\frac13}} n^{\frac13} < j_{n,1} < \xi_{n,1} < j_{n+1,1} &< n + 1 + \frac{|a_1|}{2^{\frac13}} (n+1)^{\frac13} + \frac{3}{20} a_1^2 \frac{2^{\frac13}}{(n+1)^{\frac13}} \\
        &< n + n^{\frac13} \left( n^{-\frac13} + \frac{|a_1|}{2^{\frac13}} + \frac{|a_1|}{3 \cdot 2^{\frac13} n} + \frac{3 \cdot 2^{\frac13} a_1^2}{20 n^{\frac23}}\right),
    \end{align*}
    which, assuming $n \geq 100$, can be reduced to the weaker, but simpler bound $n + 1.85 n^{\frac13} < \xi_{n,1} < n + 2.13 n^{\frac13}$.

    The $D_N$-invariant spectrum of the clamped plate problem on the disk consists of $(\xi_{0,j}^4)_{j=1}^\infty$ and $(\xi_{kN,j}^4)_{k,j=1}^\infty$. The closest points to $\lambda^2 = \xi_{n,1}$ in the spectrum are in the rotationally invariant spectrum, since $\xi_{n,j} \geq \xi_{n,2} > j_{n,2} > n + \frac{|a_2|^2}{2^{\frac13}} n^{\frac13}$ for $j \geq 2$, so $\xi_{n,j} > \xi_{n,1} + \pi$, and $\xi_{kN,j} > 2N$ for $k \geq 2$. Now, $|\xi_{n,1}^4 - \xi_{0,j}^4| \geq 4 |\xi_{n,1} - \xi_{0,j}| \min(\xi_{n,1}^3,\xi_{0,j}^3) \geq 4 n^2$ as claimed.

    If $x \in \pi \mathbb Z + (1.02,2.13)$, then $\theta_0(x) \in \pi \mathbb Z + (1.01-\frac{\pi}{4},2.13-\frac{\pi}{4})$ and $\tan(\theta_0(x)) \in (0.22,4.35)$. Combining equation \eqref{eq:W0expression} with our bounds on its constituents implies $W_0(x) \in (-1.21,-5.36)$.

    The claimed upper bound on $J_0(x)$ is simply the well-known unconditional bound on $J_0$, while the lower bound follows analogously to that on $W_0$: From $\theta_0(x) \in \pi \mathbb Z + (1.01-\frac{\pi}{4},2.13-\frac{\pi}{4})$ it follows that $\cos(\theta_0(x)) \in (0.22,0.98)$ and hence that $J_0(x) > 0.22 (\pi x/2)^{-\frac12} > 0.17 x^{-\frac12}$.
\end{proof}

\begin{lemma}
    \label{lem:bound_on_L2norms}
    Let $N \geq 100$ and $N + N^{\frac13} < \sqrt{\lambda} < N + 3N^{\frac13}$. Then
    \begin{align}
        \| J_0(\sqrt{\lambda} |\cdot|) \|_{L^2(\mathbb D_{1-\frac1N})} & \geq 0.82 \cdot \lambda^{-1/4}\\
        \| I_0(\sqrt{\lambda} |\cdot|) \|_{L^2(\mathbb D_{1-\frac1N})} & \geq 0.38 \cdot \lambda^{-1/4} I_0(\sqrt{\lambda})
    \end{align}
\end{lemma}

\begin{proof}
    The first integral can be estimated with the help of Lommel's integral (see \cite[p. 135]{Watson1966}),
    \begin{align*}
        \int_0^r s J_0(\sqrt{\lambda} s)^2 ds = \frac{r^2}{2} \left(  J_0(\sqrt{\lambda}r)^2 + J_0'(\sqrt{\lambda}r)^2 \right).
    \end{align*}
    Let $\xi$ denote the largest zero of $J_0$ less than $\sqrt{\lambda}(1-\frac1N)$, then 
    \begin{align*}
        \int_0^{1-\frac1N} s J_0(\sqrt{\lambda} s)^2 ds \geq \int_0^{\xi/\sqrt{\lambda}} s J_0(\sqrt{\lambda} s)^2 ds = \frac{\xi^2}{2\lambda} J_0'(\xi)^2.
    \end{align*}
    Consider again the modulus and phase functions, so that $J_0(x) = M_0(x)\cos(\theta_0(x))$, and recall that $\theta_0'(x) = \frac{2}{\pi x M_0(x)^2}$ (\cite[9.2.21]{AbramowitzStegun1964}). By the Wronskian identity $J_0(z)Y_0'(z)-J_0'(z)Y_0(z) = \frac{2}{\pi z}$ (\cite[9.1.16]{AbramowitzStegun1964}) and the inequality $M_0(z)^2 \leq \frac{2}{\pi z}$ (\cite[p. 447]{Watson1966}),
    \begin{align*}
        \frac{\xi^2}{2\lambda} J_0'(\xi)^2 = \frac{2}{\pi^2 \lambda Y_0(\xi)^2} = \frac{2}{\pi^2 \lambda  M_0(\xi)^2} \geq \frac{2}{\pi^2 \lambda} \frac{\xi \pi}{2} = \frac{\xi}{\pi \lambda}.
    \end{align*}
    Furthermore, the inequality on $M_0(z)$ implies $\theta'_0 > 1$ and hence that $\xi \geq \sqrt{\lambda}(1-\frac1N) - \pi$. Thus,
    \begin{align*}
        \pi \int_0^{1-\frac1N} s J_0(\sqrt{\lambda} s)^2 ds \geq \left( \frac{99}{100} - \pi \frac1{\sqrt \lambda} \right) \frac{1}{\sqrt{\lambda}} \geq \frac{0.67}{\sqrt{\lambda}},
    \end{align*}
    using the assumption $N, \sqrt{\lambda} \geq 100$.
    
    To estimate the second integral, we apply the bound $I_0(\sqrt{\lambda} r) \leq I_0(\sqrt{\lambda}) e^{\sqrt{\lambda}(r-1)}$ from \Cref{lem:bounds} and compute
    \begin{align*}
        \pi \int_0^{1-\frac1N} s I_0(\sqrt{\lambda} s)^2 ds &\geq \pi I_0(\sqrt{\lambda})^2 \int_0^{1-\frac1N} s e^{2\sqrt{\lambda}(s-1)} ds \\
        &= \frac{\pi}{2 \sqrt{\lambda}} I_0(\sqrt{\lambda})^2 \left( e^{-2 \frac{\sqrt{\lambda}}{N}} \left(1-\frac1N\right) - \frac{1}{2\sqrt{\lambda}} e^{-2 \frac{\sqrt{\lambda}}{N}} + \frac{1}{2\sqrt{\lambda}} e^{-2 \sqrt{\lambda}} \right) \\
        &\geq \frac{\pi}{2 \sqrt{\lambda}} I_0(\sqrt{\lambda})^2 \left( e^{-2 (1+3N^{-2/3})} \left(1-\frac1N\right) - \frac{1}{2N} e^{-2} \right) \geq \frac{0.15}{\sqrt{\lambda}} I_0(\sqrt{\lambda})^2.
    \end{align*}
    In the last line, the assumption $N \geq 100$ and the bounds for $\sqrt{\lambda}$, which are valid for infinitely many $N$ by \Cref{lem:nondegenerate_eigenvalues}, were used to simplify the terms occurring.
\end{proof}

\begin{lemma}
    \label{lem:remainder_bounds}
    Let $N \geq 100$ and $\sqrt{\lambda} \in (N + N^{\frac13}, N + 3 N^{\frac13})$. Let $\mathbb D'$ denote the disk of radius $1-\frac1N$. Set $\varrho(x) = \log(x)+\sqrt{1-x^2}-\log(1+\sqrt{1-x^2})$ if $x \in (0,1)$ and $0$ if $x \in (1,\infty)$. Then:
    \begin{enumerate}
        \item Let $V_N \in C^\infty(\overline{\mathbb D'})$ be given by $V_N(r,\theta) = a_N J_N(\sqrt{\lambda} r)\cos(N \theta)$ for some $a_N \in \mathbb R$. Assume that $\|V_N\|_{L^2(\mathbb D')} \leq 1$.  Then 
        \begin{align*}
            |V_N(x)| \leq 5N \exp\left(N \varrho\left( \frac{\sqrt{\lambda}}{N} |x|\right) \right)
        \end{align*}
        \item Let $V \in C^\infty(\overline{\mathbb D'})$ be a solution to $\Delta V + \lambda V = 0$ with Bessel expansion
        \begin{align*}
            V(r,\theta) = \sum_{k=2}^\infty a_k J_{kN}(\sqrt{\lambda}r) \cos(kN \theta).
        \end{align*}
        Assume that $\|V\|_{L^2(\mathbb D')} \leq 1$. Then
    \begin{align*}
        |V(r,\theta)| \leq 10 N \exp\left( 2N \varrho\left( \frac{r}{1-\frac1N} \right) \right) \left( 1 - \exp\left( N \varrho\left(  \frac{r}{1-\frac1N} \right) \right) \right)^{-2}
    \end{align*}
        \item Let $W \in C^\infty(\overline{\mathbb D'})$ be a solution to $\Delta W - \lambda W = 0$. Suppose $W$ expands into modified Bessel functions as follows:
        \begin{align*}
            W(r,\theta) = \sum_{k=1}^\infty a_k I_{kN}(\sqrt{\lambda}r) \cos(kN \theta)
        \end{align*}
        Assume $\|W\|_{L^2(\mathbb D')} \leq 1$. Then
        \begin{align*}
            |W(x)| \leq 2\sqrt{N} \left( \frac{r}{1-\frac1N} \right)^N \left( 1 - \left( \frac{r}{1-\frac1N} \right)^N \right)^{-2}
        \end{align*}
        \item Let $W_0 \in C^\infty(\overline{\mathbb D'})$ be given by $W_0(r,\theta) = b_0 I_0(\sqrt{\lambda} r)$ for some $b_0 \in \mathbb R$. Then
        \begin{align*}
            |W_0(x)| \geq I_0(\sqrt{\lambda}) e^{-\sqrt{\lambda}(1-|x|)} |W_0(0)|
        \end{align*}
    \end{enumerate}
\end{lemma}

\begin{proof}
    \begin{enumerate}
        \item 
       The assumption $\sqrt{\lambda} \geq N + N^{\frac13}$ ensures that $\frac{N}{\sqrt{\lambda}} < 1 - \frac1N$. In the following, we will often simply estimate $\sqrt{\lambda}/N \leq 2$. Now,
        \begin{multline*}
              \int_{\mathbb D'} [J_N(\sqrt{\lambda}r) \cos(N \theta)]^2 r\, dr\, d\theta \geq 
            \pi \int_0^{\frac N{\sqrt{\lambda}}} J_N\left(N \tfrac{\sqrt{\lambda}}{N} r\right)^2 r dr \\
            = \pi \left( \frac{N}{\sqrt{\lambda}} \right)^2 \int_0^1 J_N\left(N s \right)^2 s ds 
            \geq \frac{e^{-\frac1{6N}}}{2N} \left( \frac{N}{\sqrt{\lambda}} \right)^2 \int_0^1 s^{2N+1} ds \geq \frac{1}{25 N^2}.
        \end{multline*}
        The bound on $V(r,\theta) = a_n J_n(\sqrt{\lambda}r) \cos(n\theta)$ now follows from the lower bound in \Cref{lem:bounds} and the assumption $\|V\|_{L^2(\mathbb D')} \leq 1$.
        \item The lower bound $J_n(\sqrt{\lambda} r) \geq e^{-\frac1{12n}} \frac{1}{\sqrt{2 \pi n}} \exp(n \varrho(\frac{\sqrt{\lambda}}{n} r))$ from Lemma~\ref{lem:bounds} lets us estimate
        \begin{align*}
            \int_{\mathbb D'} [J_n(\sqrt{\lambda} r) \cos(n \theta)]^2 r dr d\theta &= \pi \int_0^{1-\frac1N} J_n\left(n \frac{\sqrt{\lambda}}{n} r\right)^2 r dr \\
            &\geq  e^{-\frac1{6n}} \frac{1}{2n} \left(\frac{n}{\sqrt{\lambda}}\right)^2 \int_0^{\tfrac{\sqrt{\lambda}}{n}(1-\frac1N)} \exp(2n \varrho(s)) s ds.
        \end{align*}
        This integral can be estimated using the change of variables $t = \varrho(s)$, keeping in mind that $\frac{e}{2} s \geq e^t = s \exp(\sqrt{1-s^2}-\log(1+\sqrt{1-s^2})) \geq s$:
        \begin{align*}
            e^{-\frac1{6n}} \frac{1}{2n} & \left(\frac{n}{\sqrt{\lambda}}\right)^2 \int_{-\infty}^{\varrho(\frac{\sqrt{\lambda}}{n} (1-\frac1N))} e^{2nt} \left(1-s^2\right)^{-\frac12} s^2 dt \\ &\geq \left(\frac{2}{e}\right)^2 e^{-\frac1{6n}} \frac{1}{2n(2n+2)} \left(\frac{n}{\sqrt{\lambda}}\right)^2 \exp((2n+2) \varrho(\tfrac{\sqrt{\lambda}}{n} (1-\tfrac1N)))) \\
            & \geq \left(\frac{2}{e}\right)^2 e^{-\frac1{6n}} \frac{1}{2n(2n+2)} \left(1-\tfrac1N\right)^2 \exp(2n \varrho(\tfrac{\sqrt{\lambda}}{n} (1-\tfrac1N)))) \\
            & \geq \frac{1}{10 n^2} \exp\left(2n \varrho(\tfrac{\sqrt{\lambda}}{n} (1-\tfrac1N))\right)
        \end{align*}
        Since $\|a_n J_n(\sqrt{\lambda}r) \cos(n\theta)\|_{L^2(\mathbb D')} \leq 1$ for $n=kN$ because
        \[
        \|V\|_{L^2(\mathbb D')}^2=\sum_{k=2}^\infty |a_k|^2 \int_{\mathbb D'} [J_n(\sqrt{\lambda} r) \cos(n \theta)]^2 r dr d\theta,
        \]
        the upper bound in Lemma~\ref{lem:bounds} then implies 
        \begin{align*}
            |a_n J_n(\sqrt{\lambda}r) \cos(n\theta)| \leq \sqrt{10} n \exp\left(n (\varrho(\tfrac{\sqrt{\lambda}}{n} r) - \varrho(\tfrac{\sqrt{\lambda}}{n} (1-\tfrac1N)))\right)
        \end{align*}
        It is easy to check that $\varrho(tr)-\varrho(tR)$ increases in $t$ for any $0 < r < R$ and that under our assumptions on $N$ and $\sqrt{\lambda}$ we have $\frac{\sqrt{\lambda}}{n} (1-\frac1N) \leq 1$ for $n \in \{2N,3N,\ldots\}$, hence 
        \begin{align*}
            |a_{kN} J_{kN}(\sqrt{\lambda}r) \cos(kN\theta)| \leq \sqrt{10} {kN} \exp\left(kN \varrho((1-\tfrac1N)^{-1} r) \right)
        \end{align*}
        Summing up the estimates on the terms in the Bessel expansion of $V$ yields the claimed bound.
        \item Let $\varrho_+(t) = \log(t) + \sqrt{1+t^2} - \log(1+\sqrt{1+t^2})$. From \Cref{lem:bounds} we obtain 
        \begin{align*}
            \int_{\mathbb D'} &[I_n(\sqrt{\lambda} r) \cos(n \theta)]^2 r dr d\theta = \pi \int_0^{1-\frac1N} I_n\left(n \frac{\sqrt{\lambda}}{n} r\right)^2 r dr \\
            \geq \ &\pi (2\pi n)^{-1} \int_0^{1-\frac1N} \exp\left( 2n \varrho_+ \left(\tfrac{\sqrt{\lambda}}{n} r\right) \right) \left( \frac{n+1}n + \left(\frac{\sqrt{\lambda}}{n} r\right)^2 \right)^{-\frac14} rdr \\
            \geq & \ (2 n)^{-1} \left( \frac{51}{10} \right)^{-\frac14} \left(\frac{n}{\sqrt{\lambda}} \right)^2
            \int_0^{\frac{\sqrt{\lambda}}{n}\left(1-\frac1N\right)} \exp\left(2n\varrho_+ \left( s \right)\right) sds
        \end{align*}
        We estimate the integral in the last line using the change of variables $t = \varrho_+ (s)$. Recall that $\frac{dt}{ds} = \frac{\sqrt{1+s^2}}{s}$ and $e^t = s e^{\sqrt{1+s^2} - \log(1 + \sqrt{1+s^2})}$. It follows that
        \begin{align*}
            &\int_0^{\frac{\sqrt{\lambda}}{n}\left(1-\frac1N\right)} \exp\left(2n\varrho_+ \left( s \right)\right) sds \\ =
            \ &\int_{-\infty}^{\varrho_+\left( \frac{\sqrt{\lambda}}{n}\left(1-\frac1N\right) \right)} \exp((2n+2) t) \frac{e^{2 \log(1 + \sqrt{1+s^2}) - 2 \sqrt{1+s^2}}}{\sqrt{1+s^2}} dt \\
            \geq \ & \frac{e^{2 \log(1 + \sqrt{1+s^2}) - 2 \sqrt{1+s^2}}}{\sqrt{1+s^2}}\bigg|_{s = \frac{\sqrt{\lambda}}{n}\left(1-\frac1N\right)} \frac{1}{2n+2} \exp\left((2n+2) \varrho_+ \left( \frac{\sqrt{\lambda}}{n} \left( 1 - \frac1N \right)\right) \right)
            \\ \geq \ & \frac{1}{\sqrt{5}} \left( \frac{\sqrt{\lambda}}{n}\right)^2 \left(1-\frac1N\right)^2 \frac{1}{2n+2} \exp\left(2n \varrho_+ \left(\frac{\sqrt{\lambda}}{n} \left( 1 - \frac1N \right)\right)\right)
        \end{align*}
    Simplifying the constants using the assumption $n \geq N \geq 10$ and $\sqrt{\lambda}/n \leq 2$ yields
    \begin{align*}
        \int_{\mathbb D} [I_n(\sqrt{\lambda} r) \cos(n \theta)]^2 r dr d\theta \geq \frac1{20 n^2} \exp\left(2n \varrho_+ \left(\frac{\sqrt{\lambda}}{n} \left( 1 - \frac1N \right)\right)\right)
    \end{align*}
    Thus, just as in (2), the fact that $\|a_n I_n(\sqrt{\lambda}r) \cos(n \theta)\|_{L^2(\mathbb D')} \leq 1$ with $n=kN$ and the upper bound in Lemma~\ref{lem:bounds} imply 
    \begin{align*}
        |a_n I_n(\sqrt{\lambda}r) \cos(n \theta)| &\leq \frac{\sqrt{20}n e^{\frac{1}{6n}}}{\sqrt{2\pi n}} \exp\left(n \varrho_+ \left(\frac{\sqrt{\lambda}}{n} r\right) - n \varrho_+ \left(\frac{\sqrt{\lambda}}{n} \left( 1 - \frac1N \right)\right) \right) \\
        &\leq 2\sqrt{n} \left( \frac{r}{1-\frac1N} \right)^n.
    \end{align*}
    The claimed bound now follows by summing up the bounds as
    \begin{align*}
        \bigg|\sum_{k=1}^\infty & a_{kN} I_{kN}(\sqrt{\lambda}r) \cos(kN \theta)\bigg| \leq2\sqrt N\sum_{k=1}^\infty \sqrt{k} \left( \frac{r}{1-\frac1N} \right)^{kN}\\
        &\leq2\sqrt N\sum_{k=1}^\infty k \left( \frac{r}{1-\frac1N} \right)^{kN}=2\sqrt{N} \left( \frac{r}{1-\frac1N} \right)^N \left( 1 - \left( \frac{r}{1-\frac1N} \right)^N \right)^{-2}.
    \end{align*}
    \item This bound follows directly from Lemma~\ref{lem:bounds}.
    \end{enumerate}
\end{proof}

\section{Estimates for derivatives of eigenvalues and eigenfunctions}

\subsection{Abstract estimates}
This section is dedicated to obtaining explicit derivative estimates on the eigenvalues and eigenfunctions of a spectral problem varying along a finite number of parameters. 

The underlying setup is an adaptation of Kato's operator families of type II. Since the application we have in mind is the clamped plate problem on a smoothly bounded domain, we will exploit its peculiarities quite strongly in order to increase readability of the presentation. However, it will be clear from their proof that it would be straightforward, although perhaps tedious, to apply the estimates obtained below to a variety of elliptic boundary value problems.
    
    \begin{lemma}
    \label{lem:abstractEstimates}
    Let $X$ and $H$ be real Hilbert spaces, together with a compact inclusion $\iota: X \to H$. 
    Let $A_{(\cdot)} \in C^s([-1,1]^m, \mathcal{L}(X,H))$ and $B_{(\cdot)} \in C^s([-1,1]^m, \mathcal{L}(H,H^\ast))$ for some $m,s \in \mathbb N$. Assume $B_t^\ast = B_t$ for all $t \in [-1,1]^m$, $\left\langle B_0 (\cdot) , \cdot \right\rangle = (\cdot , \cdot )_H$ and $\|\!\cdot\!\|_X = \|A_0 (\cdot) \|_H$.
    
    Consider $\lambda_0 \geq 0$ such that $\ker( \lambda_0^2 \iota^\ast B_0 \iota - A_0^\ast B_0 A_0 )$ is spanned by a single vector $u_0 \in X$, which we normalize to satisfy $\|\iota u_0\|_H = 1$. Let $\tau$ be such that $\ker(\xi \iota^\ast B_0 \iota - A_0^\ast B_0 A_0) = \{0\}$ for all $\xi \in (\lambda_0^2 - \tau,\lambda_0^2 + \tau) \setminus \{\lambda^2\}$. 

    There exists $\varepsilon > 0$ such that $\lambda_0$ and $u_0$ extend uniquely to functions $\lambda_{(\cdot)} \in C^s([-\varepsilon,\varepsilon]^m)$ and $u_{(\cdot)} \in C^s([-\varepsilon,\varepsilon]^m, X)$ satisfying $(\lambda_t^2 \iota^\ast B_t \iota - A_t^\ast B_t A_t) u_t = 0$ and $\left\langle B_t u_t, u_t \right\rangle = 1$ for all $t \in [-\varepsilon,\varepsilon]^m$. At $t=0$,
    \begin{align}
        \label{eq:firstDerivativeEstimate}
        \sum_{j=1}^m |\partial_{t_j} (\lambda_t^2)| \leq 2 \lambda_0^2 \sum_{j=1}^m (\| \partial_{t_j} A_t \|_{X \to H} + \| \partial_{t_j} B_t \|_{H \to H^\ast}).
    \end{align}
    Let $\rho > 0$ be small enough so that at $t=0$,
    \begin{align}
        \label{eq:abstractConditionA}
        \sum_{|\alpha| \leq s} \frac{(2\rho)^{|\alpha|}}{\alpha!} \|\partial_t^\alpha (A_t-A_0) \|_{X \to H} \leq \frac{\tau}{20(\lambda_0^2 + \tau)} \\
        \label{eq:abstractConditionB}
        \sum_{|\alpha| \leq s} \frac{(2\rho)^{|\alpha|}}{\alpha!} \|\partial_t^\alpha (B_t-B_0) \|_{H \to H^\ast} \leq \frac{\tau}{50(\lambda_0^2 + \tau)}
    \end{align}
    Then the following estimates hold at $t=0$:
    \begin{align}
        \sum_{|\alpha| \leq s} \frac{\rho^{|\alpha|}}{\alpha!} |\partial_t^\alpha (\lambda_t^2 - \lambda_0^2)| & \leq 2\tau, &
        \sum_{|\alpha| \leq s} \frac{\rho^{|\alpha|}}{\alpha!} \|\partial_t^\alpha u_t\|_X & \leq 2(\lambda_0^2 + \tau)^{\frac12}, &
        \sum_{|\alpha| \leq s} \frac{\rho^{|\alpha|}}{\alpha!} \|\iota \partial_t^\alpha u_t\|_H & \leq 2.
    \end{align}
\end{lemma}

In our application, the Hilbert spaces we shall consider will be $X=H^2_{00}(\mathbb D)$ and $H=L^2(\mathbb D)$; $A_t$ will be Laplacian defined by a family of metrics parametrized by~$t$, and $B_t$ will be the multiplication by a scalar function. The proof of \Cref{lem:abstractEstimates} consists of a bootstrap argument resting on the Hellmann--Feynman formula, by which one can express the $k^{th}$ derivative of the eigenvalue using only lesser derivatives of the eigenfunctions. Compared to complex-analytic methods, this approach has the benefit of yielding improved estimates ---sharp up to a constant--- on the derivatives of the eigenfunction when measured in $H$.

\begin{proof}
    In the notation of this proof, we will suppress $t=0$ consistently, e.g.\ writing $A$ for $A_0$. We will also write $\partial^\alpha u$ for $\partial_t^\alpha|_{t = 0} u_t$, $\partial^\alpha A$ for $\partial_t^\alpha|_{t = 0} A_t$ and so on.

    Let us recall a few basic functional analytic properties that will be crucial in what follows. The assumption $\| A u \|_{H} = \|u\|_X$ implies that the map $A^\ast B A:X\to X^*$ is bounded below, hence injective with closed range. Since $A^\ast B A$ is its own adjoint, it is also surjective.
    Now, $\iota (A^\ast B A)^{-1} \iota^\ast B$ is a positive, compact, self-adjoint operator on $H$. Its spectrum thus consists of discrete eigenvalues $(\mu_j)_{j=1}^\infty \subseteq (0,\infty)$ to which one may associate a corresponding orthonormal eigenbasis $(w_j)_{j=1}^\infty$. Let $\lambda_j^2 := \mu_j^{-1}$ and $u_j = \mu_j^{-1} (A^\ast B A)^{-1} \iota^\ast B w_j$, so that $(\lambda_j^2 \iota^\ast B \iota - A^\ast B A) u_j = 0$ and $(\iota u_j)_{j=1}^\infty = (w_j)_{j=1}^\infty$ forms an orthonormal basis of $H$. Furthermore, $\left( A u_j, A u_k \right)_H = \delta_{jk} \lambda_j^2$, so $(\lambda_j^{-1} A u_j)_{j=1}^\infty$ forms an orthonormal basis of $\mathrm{Ran}(A)$. Thus, $\|v\|_X^2 = \sum_{j=1}^\infty \lambda_j^{-2} (A v, A u_j)_H^2 = \sum_{j=1}^\infty \lambda_j^2 (\iota v, \iota u_j)_H^2$ for every $v \in X$, which implies that $\|w\|_{X^\ast}^2 = \sum_{j=1}^\infty \lambda_j^{-2} \left\langle w, u_j \right\rangle^2$ for every $w \in X^\ast$. Hence, the sum $\sum_{j=1}^\infty (\iota v,\iota u_j)_H u_j$ converges strongly in $X$ to $v$, and for every $w \in X^\ast$, $\sum_{j=1}^\infty \left\langle w, u_j\right\rangle \iota^\ast \iota u_j$ converges strongly in $X^\ast$ to $w$. Thus, the operators $P_{u^\bot}^X$, $P_{u^\bot}^H$ and $P_{u^\bot}^{X^\ast}$ given by 
    \begin{align*}
        P_{u^\bot}^X v &= \sum_{\lambda_j \neq \lambda} (\iota v, \iota u_j)_H u_j, & P_{u^\bot}^H y &= \sum_{\lambda_j \neq \lambda} (y, \iota u_j)_H \iota u_j, & P_{u^\bot}^{X^\ast} w &= \sum_{\lambda_j \neq \lambda} \left\langle w, u_j\right\rangle_{X^\ast} \iota^\ast B \iota u_j,
    \end{align*}
    are projections of norm $1$ which intertwine $\iota$, $\iota^\ast$ and the operator of interest: $\iota P_{u^\bot}^X = P_{u^\bot}^H \iota$, $\iota^\ast B P_{u^\bot}^{H} = P_{u^\bot}^{X^\ast} \iota^\ast$ and $P_{u^\bot}^{X^\ast} (\mu^2 \iota^\ast B \iota - A^\ast B A) = (\mu^2 \iota^\ast B \iota - A^\ast B A) P_{u^\bot}^{X}$. We will also adopt the notation $P_u^X = \mathds 1_X - P_{u^\bot}^{X}$ for the projection onto $u$.
    
    Next, we will prove the following four claims.
    \begin{enumerate}[label = (\roman*)]
        \item For $v \in X$, $\|P_{u^\bot}^X v\|_X \leq \frac{\lambda^2 + \tau}{\tau} \|(\lambda^2 \iota^\ast B \iota - A^\ast B A) v\|_{X^\ast}$.
        \item For $v \in X$, $\|\iota P_{u^\bot}^X v\|_H \leq \frac{\sqrt{\lambda^2 + \tau}}{\tau} \|(\lambda^2 \iota^\ast B \iota - A^\ast B A) v\|_{X^\ast}$.
        \item If $v \in X$ and $(\lambda^2 \iota^\ast B \iota - A^\ast B A) v = \iota^\ast B w$ for some $w \in H$, then $\|\iota P_{u^\bot}^X v\|_H \leq \frac{1}{\tau} \|w\|_{H}$.
        \item If $v \in X$ and $(\lambda^2 \iota^\ast B \iota - A^\ast B A) v = \iota^\ast B w$ for some $w \in H$, then $\|P_{u^\bot}^X v\|_X \leq \frac{\sqrt{\lambda^2 + \tau}}{\tau}  \|w\|_{H}$.
    \end{enumerate}
    To prove the first two claims, let $\tilde v = \sum_{\lambda_j \neq \lambda} \frac{|\lambda_j^2 - \lambda^2|}{\lambda^2 - \lambda_j^2} (\iota v, \iota u_j)_H u_j$. Then
    \begin{align*}
        \|(\lambda^2 \iota^\ast B \iota - A^\ast B A) v\|_{X^\ast} &\geq \|\tilde v \|_X^{-1} \left\langle (\lambda^2 \iota^\ast B \iota - A^\ast B A) v, \tilde v \right\rangle \\
        &= \|\tilde v \|_X^{-1} \left( \lambda^2 (\iota v, \iota \tilde v)_H - (A v, A \tilde v)_H \right)\\
        &= \|\tilde v \|_X^{-1} \sum_{\lambda_j \neq \lambda} (\iota v, \iota u_j)_H^2 |\lambda_j^2 - \lambda^2|.
    \end{align*}
    Now we use that $\|\tilde v \|_X^2 = \| P_{u^\bot}^X v \|_X^2 = \sum_{\lambda_j \neq \lambda} (\iota v, \iota u_j)_H^2 \lambda_j^2$ to obtain
    \begin{align*}
        \|(\lambda^2 \iota^\ast B \iota - A^\ast B A) v\|_{X^\ast} &\geq \| \tilde v \|_X^{-1} \left(\inf_{\lambda_i \neq \lambda} \frac{|\lambda_i^2 - \lambda^2|}{\lambda_i^2}\right) \sum_{\lambda_j \neq \lambda} (\iota v, \iota u_j)^2 \lambda_j^2 \geq \frac{\tau}{\lambda^2 + \tau} \|P_{u^\bot}^X v\|_X,
    \end{align*}
    establishing claim (i). Claim (ii) now follows, keeping in mind that $\|\tilde v \|_X = \| P_{u^\bot}^X v \|_X$:
    \begin{align*}
        \|(\lambda^2 \iota^\ast B \iota - A^\ast B A) v\|_{X^\ast}^2 &\geq \frac{\tau}{\lambda^2 + \tau} \|(\lambda^2 \iota^\ast B \iota - A^\ast B A) v\|_{X^\ast} \|\tilde v \|_X \\
        & \geq \frac{\tau}{\lambda^2 + \tau} \left( \inf_{\lambda_i \neq \lambda} |\lambda_i^2 - \lambda^2| \right) \sum_{\lambda_j \neq \lambda} (\iota v, \iota u_j)^2 
        = \frac{\tau^2}{\lambda^2 + \tau} \| \iota P_{u^\bot}^X v \|_H^2.
    \end{align*}
    For claim (iii), note that $A^\ast B A v = A^\ast B A \sum_{j=1}^\infty (\iota v, \iota u_j)_H u_j = \sum_{j=1}^\infty \lambda_j^2 (\iota v, \iota u_j)_H \iota^\ast B \iota u_j$, so
    \begin{align*}
        (w,\iota u_k)_H &= \left\langle (\lambda^2 \iota^\ast B \iota - A^\ast B A) v , u_k \right\rangle = \left\langle \sum_{j=1}^\infty (\iota v, \iota u_j)_H (\lambda^2 - \lambda_j^2) \iota^\ast B \iota u_j, u_k \right\rangle \\
        &= (\iota v, \iota u_k)_H (\lambda^2 - \lambda_k^2),
    \end{align*}
    which, summing up the squares, implies $\|w\|_H \geq \tau \|\iota P_{u^\bot}^X v\|_H$. Claim (iv) follows similarly:
    \begin{align*}
        \|w\|_H^2 = \sum_{j=1}^\infty (w,\iota u_k)_H^2 \geq \left( \inf_{\lambda_i \neq \lambda} \frac{|\lambda^2 - \lambda_i^2|}{\lambda_i} \right)^2 \sum_{j=1}^\infty \lambda_j^2 (\iota v,\iota u_k)_H^2 \geq \frac{\tau^2}{\lambda^2 + \tau} \|P^X_{u^\bot} v\|_X^2
    \end{align*}
    
    We make the following definitions: For $0 \leq k \leq s$,
    \begin{align*}
        a_k(\rho) &= \frac{\lambda^2 + \tau}{\tau} \sum_{|\alpha| \leq k} \frac{|\alpha|}{\alpha!}\rho^{|\alpha|} \|\partial^\alpha (A^\ast B A) \|_{X \to X^\ast}, & b_k(\rho) &= \frac{\lambda^2 + \tau}{\tau} \sum_{|\alpha| \leq k} \frac{|\alpha|}{\alpha!}\rho^{|\alpha|} \|\partial^\alpha B\|_{H \to H^\ast}, \\
        x_k(\rho) &= 1 + \frac{1}{\lambda^2+\tau} \sum_{1 \leq |\alpha| \leq k} \frac{\rho^{|\alpha|}}{\alpha!} | \partial^\alpha (\lambda^2) |, & y_k(\rho) &= 1 + \frac{1}{\sqrt{\lambda^2+\tau}} \sum_{1 \leq |\alpha| \leq k} \frac{\rho^{|\alpha|}}{\alpha!} \|\partial^\alpha u\|_{X}, \\
        z_k(\rho) &= 1 + \sum_{1 \leq |\alpha| \leq k} \frac{\rho^{|\alpha|}}{\alpha!} \|\iota (\partial^\alpha u) \|_{H}
    \end{align*}
    These quantities are set up to satisfy the following recursive inequalities:
        \begin{align}
            \label{eq:recursion_x}\tag{v}
            x_{k+1}(\rho) &\leq 1 + \frac{\tau}{\lambda^2 + \tau} \left( a_{k+1}(\rho) y_k(\rho)^2 + b_{k+1}(\rho) x_k(\rho) z_k(\rho)^2 \right), \\
            \label{eq:recursion_y}\tag{vi}
        \begin{split}
        y_{k+1}(\rho) &\leq x_{k+1}(\rho) b_{k+1}(\rho) z_k(\rho) + a_{k+1}(\rho) y_k(\rho) + a_{k+1}(\rho) y_k(\rho)^2 z_k(\rho) + b_{k+1}(\rho) x_k(\rho) z_k(\rho)^3 \\ &+ \frac12 b_{k+1}(\rho) z_k(\rho)^2 + \frac12 (z_k(\rho)-1)^2 + 1,
        \end{split} \\
            \label{eq:recursion_y}\tag{vii}
        \begin{split}
        z_{k+1}(\rho) &\leq x_{k+1}(\rho) b_{k+1}(\rho) z_k(\rho) + a_{k+1}(\rho) y_k(\rho) + a_{k+1}(\rho) y_k(\rho)^2 z_k(\rho) + b_{k+1}(\rho) x_k(\rho) z_k(\rho)^3 \\ &+ \frac12 b_{k+1}(\rho) z_k(\rho)^2 + \frac12 (z_k(\rho)-1)^2 + 1,
        \end{split} 
        \end{align}
    Claim (v) follows from the Hellmann--Feynman formula, which in this setting takes the form
    \begin{align*}
         \partial_j (\lambda^2) = \left\langle \partial_j (A^* B A) u, u\right\rangle - \lambda^2 \left\langle (\partial_j B) \iota u, \iota u\right\rangle .
    \end{align*}
    Indeed, since $\left\langle \iota^\ast B_t \iota u_t, u_t\right\rangle \equiv 1$ by assumption, $\lambda_t^2 = \left\langle A_t^\ast B_t A_t u_t, u_t \right\rangle$, so
    \begin{align*}
        \partial_{j}(\lambda^2) &= \left\langle \partial_{j}(A^\ast BA) u, u \right\rangle + 2 \left\langle A^\ast B A u, \partial_j u \right\rangle \\
        &= \left\langle \partial_{j}(A^\ast BA) u, u \right\rangle + 2 \lambda^2 \left\langle B \iota u, \iota (\partial_j u) \right\rangle \\
        &= \left\langle \partial_{j}(A^\ast BA) u, u \right\rangle - \lambda^2 \left\langle \partial_j B \iota u, \iota u \right\rangle.
    \end{align*}
    Observe that $\sum_{j} \frac{1}{(\alpha-e_j)!} = \frac{|\alpha|}{\alpha!}$. Hence,
    \begin{align*}
        \frac{1}{\tau} &\sum_{|\alpha| \leq k+1} \frac{|\alpha|}{\alpha!}\rho^{|\alpha|} |\partial^\alpha (\lambda^2)| = \frac{1}{\tau} \sum_j \sum_{|\alpha|\leq k} \frac{\rho^{|\alpha|+1}}{\alpha!} |\partial^\alpha \partial_j (\lambda^2)| \\
        \leq \ & \frac{\lambda^2 + \tau}{\tau} \sum_j \sum_{|\alpha|\leq k} \sum_{\beta + \gamma + \eta = \alpha} \frac{\rho^{|\beta|+1}}{\beta!} \|\partial^\beta \partial_j (A^\ast B A)\|_{X \to X^\ast} \frac{1}{\sqrt{\lambda^2+\tau}} \frac{\rho^{|\gamma|}}{\gamma!} \|\partial^\gamma u\|_X \frac{1}{\sqrt{\lambda^2+\tau}} \frac{\rho^{|\eta|}}{\eta!} \|\partial^\eta u\|_X \\
        + \ &\frac{\lambda^2 + \tau}{\tau} \sum_j \sum_{|\alpha|\leq k} \sum_{\beta + \gamma + \eta + \nu = \alpha} \frac{\rho^{|\beta|+1}}{\beta!} \|\partial^\beta \partial_j B\|_{H \to H^\ast} \frac{\rho^{|\gamma|}}{\gamma!} \|\iota \partial^\gamma u\|_H \frac{\rho^{|\eta|}}{\eta!} \|\iota \partial^\eta u\|_H \frac{1}{\lambda^2 + \tau} \frac{\rho^{|\nu|}}{\nu!} |\partial^\nu(\lambda^2)| \\
        \leq \ &\left( \frac{\lambda^2 + \tau}{\tau} \sum_{1 \leq |\alpha|\leq k+1} \frac{|\alpha|}{\alpha!} \rho^{|\alpha|} \|\partial^{\alpha}(A^\ast B A)\|_{X\to X^*} \right) \left( \frac{1}{\sqrt{\lambda^2 + \tau}} \sum_{|\beta|\leq k} \frac{\rho^{|\beta|}}{\beta!} \|\partial^{\beta} u \|_X \right)^2 \\
        + \ &\left( \frac{\lambda^2 + \tau}{\tau} \sum_{1 \leq |\alpha|\leq k+1} \frac{|\alpha|}{\alpha!} \rho^{|\alpha|} \|\partial^{\alpha}B\|_{H \to H^\ast} \right) \left( \sum_{|\beta|\leq k} \frac{\rho^{|\beta|}}{\beta!} \|\iota \partial^{\beta} u \|_H \right)^2 \left(\frac{1}{\lambda^2 + \tau} \sum_{|\gamma|\leq k} \frac{\rho^{|\gamma|}}{\gamma!} |\partial^{\gamma} (\lambda^2) | \right)
    \end{align*}
    Since $|\alpha| \geq 1$, the sum on the left hand side bounds $\frac{\lambda^2+\tau}{\tau}(x_{k+1}(\rho) - 1)$. As $\|u\|_X = \lambda$, $\|u\|_H = 1$, the right hand side can be weakened to yield (v).
    
    To obtain (vi), we first derive a bound on $\|P_{u^\bot}^X \partial^{\alpha} u \|_X$. By the product rule,
    \begin{align*}
        - (\lambda^2 \iota^\ast B \iota - A^\ast B A) \partial^{\alpha} u &=
        \sum_{\substack{\beta + \gamma = \alpha \\ |\gamma| \leq |\alpha|-1}} \binom{\alpha}{\gamma,\beta} \partial^\beta \left( (\lambda^2 \iota^\ast B \iota - A^\ast B A) \right) \partial^{\gamma} u \\
        &= \sum_{\substack{\beta + \gamma +\eta = \alpha \\ |\eta| \leq |\alpha|-1}} \binom{\alpha}{\gamma,\beta,\eta} \partial^\beta (\lambda^2) \iota^\ast \partial^\gamma B \iota \partial^\eta u + \sum_{\substack{\beta + \gamma = \alpha \\ |\gamma| \leq |\alpha|-1}} \binom{\alpha}{\gamma,\beta} \partial^\beta(A^\ast B A) \partial^\gamma u.
    \end{align*}
    We apply $P_{u^\bot}^{X^\ast}$ to both sides, observe that $(\lambda \iota^\ast B \iota - A^\ast B A): \mathrm{Ran}(P_{u^\bot}^{X}) \to \mathrm{Ran}(P_{u^\bot}^{X^\ast})$ is invertible, and invoke (i) and (iv):
    \begin{align*}
        \frac{1}{\sqrt{\lambda^2 + \tau}} \frac{\rho^{|\alpha|}}{\alpha!} \| P_{u^\bot}^X \partial^{\alpha} u \|_{X} &\leq \frac{1}{\sqrt{\lambda^2 + \tau}}\frac{\sqrt{\lambda^2 + \tau}}{\tau} \sum_{\substack{\beta + \gamma +\eta = \alpha \\ |\eta| \leq |\alpha|-1 }} \frac{\rho^{|\beta|}}{\beta!} | \partial^\beta (\lambda^2) | \frac{\rho^{|\gamma|}}{\gamma!} \| \partial^\gamma B \|_{H \to H^\ast} \frac{\rho^{|\eta|}}{\eta!} \| \iota \partial^\eta u \|_{H} \\
        &+ \frac{1}{\sqrt{\lambda^2 + \tau}}\frac{\lambda^2 + \tau}{\tau} \sum_{\substack{\beta + \gamma = \alpha \\ |\gamma| \leq |\alpha|-1}} \frac{\rho^{|\beta|}}{\beta!} \| \partial^\beta(A^\ast B A)\|_{X \to X^\ast} \frac{\rho^{|\gamma|}}{\gamma!} \| \partial^\gamma u \|_X \\
        &\leq \sum_{\substack{\beta + \gamma +\eta = \alpha \\ |\eta| \leq |\alpha|-1 }} \frac{1}{\lambda^2+\tau} \frac{\rho^{|\beta|}}{\beta!} | \partial^\beta (\lambda^2) | \frac{\lambda^2 + \tau}{\tau} \frac{\rho^{|\gamma|}}{\gamma!} \| \partial^\gamma B \|_{H \to H^\ast} \frac{\rho^{|\eta|}}{\eta!} \| \iota \partial^\eta u \|_{H} \\
        &+ \sum_{\substack{\beta + \gamma = \alpha \\ |\gamma| \leq |\alpha|-1}} \frac{\lambda^2 + \tau}{\tau} \frac{\rho^{|\beta|}}{\beta!} \| \partial^\beta(A^\ast B A)\|_{X \to X^\ast} \frac{1}{\sqrt{\lambda^2 + \tau}} \frac{\rho^{|\gamma|}}{\gamma!} \| \partial^\gamma u \|_X \\
    \end{align*}
    Summing over all $\alpha$ with $1 \leq |\alpha| \leq k$, and keeping in mind that $P_{u^\bot}^X u = 0$, this yields
    \begin{align}
        \label{eq:yRecursionPartI}
        \begin{split}
        \frac{1}{\sqrt{\lambda^2 + \tau}} &\sum_{|\alpha| \leq k} \frac{\rho^{|\alpha|}}{\alpha!} \| P_{u^\bot}^X \partial^{\alpha} u \|_{X} \\ &\leq x_{k}(\rho) b_k(\rho) z_{k-1}(\rho) + \frac{\lambda^2 + \tau}{\tau} (x_{k}(\rho)-1) z_{k-1}(\rho) + a_{k}(\rho) y_{k-1}(\rho).
        \end{split}
    \end{align}
    Replacing $\frac{\lambda^2 + \tau}{\tau} (x_{k}(\rho)-1)$ with the recursive estimate (v) yields one half of the terms in (vi). 
    
    The component of $\partial^\alpha u$ parallel to $u$ can be obtained as follows: Since $\left\langle B_t u_t, u_t \right\rangle \equiv 1$,
    \begin{align*}
        (\iota \partial^\alpha u, \iota u)_H = - \frac12 \sum_{\substack{\beta + \gamma + \eta = \alpha \\ |\gamma|,|\eta| < |\alpha|}} \left\langle \partial^\beta B \, \iota \partial^\gamma u, \iota \partial^\eta u \right\rangle,
    \end{align*}
    for any $|\alpha| \geq 1$, and hence 
    \begin{align*}
        \frac{1}{\sqrt{\lambda^2 + \tau}} \frac{\rho^{|\alpha|}}{\alpha!} \|P_u^X \partial^\alpha u\|_X \leq \frac{1}{2} \frac{ \| u \|_X}{\sqrt{\lambda^2 + \tau}} \sum_{\substack{\beta + \gamma + \eta = \alpha \\ |\gamma|,|\eta| < |\alpha|}} \frac{\rho^{|\beta|}}{\beta!}  \| \partial^\beta B \|_{H \to H^\ast} \frac{\rho^{|\gamma|}}{\gamma!} \| \iota \partial^\gamma u \|_H \frac{\rho^{|\eta|}}{\eta!} \| \iota \partial^\eta u \|_H,
    \end{align*}
    Summing over $\alpha$, we find that 
    \begin{align}
        \label{eq:yRecursionPartII}
        1 + \frac{1}{\sqrt{\lambda^2 + \tau}} \sum_{1 \leq |\alpha| \leq k} \frac{\rho^{|\alpha|}}{\alpha!} \|P_u^X \partial^\alpha u\|_X \leq 1 + \frac12 b_k(\rho) z_{k-1}(\rho)^2 + \frac12 (z_{k-1}(\rho)-1)^2.
    \end{align}
    By the triangle inequality, \eqref{eq:yRecursionPartI} and \eqref{eq:yRecursionPartII} and claim (v) together imply claim (vi). Claim (vii) follows completely analogously, using (ii) and (iii) instead of (i) and (iv).

    The recursive inequalities (vi) and (vii) are no longer dependent on $\tau$ or $\lambda$, and the factor $\frac{\tau}{\lambda^2 + \tau}$ appearing in (v) is bounded by $1$. Furthermore, $x_0(\rho) = y_0(\rho) = z_0(\rho) = 1$. We will obtain good estimates on $x_s(\rho)$, $y_s(\rho)$ and $z_s(\rho)$ by a bootstrap argument. Assume $a_k(\rho) \leq A$, $b_k(\rho) \leq B$ for all $1 \leq k \leq s$. Assume furthermore that $y_k(\rho), z_k(\rho) \leq 2$ and $x_k(\rho) \leq \frac75$ for all $k \in \{0,\ell\}$ and some $\ell \in \{0,s-1\}$ (for $\ell = 0$, this is trivially satisfied). We will show that the same conditions then hold for $\ell+1$.
    
    Let $m_\ell$ be recursively defined by $m_0 = 1$, 
    \begin{align*}
        m_{k+1} \leq \left(\frac75 B + A\right)(m_k + m_k^3) + \tfrac 12 B m_k^2 + \tfrac 12 (m_k-1)^2 + 1, \ k \in \{0,\ldots,\ell\}.
    \end{align*}
    If $A \leq \frac1{17}, B \leq \frac1{100}$, then the graph of $f(x) = (\frac75 B + A)(x+x^3) + \frac B2x^2 + \frac12 (x-1)^2 + 1$ intersects $\{y=x\}$ at a point $y_\infty \in (0,2)$, and $y_\infty > f(x) > x$ for all $x \in [1,y_\infty)$. It follows that $m_{k+1} \leq y_\infty < 2$ for all $k \in \{0,\ldots,\ell\}$.

    Our bootstrap assumptions together with (v) and (vi) imply that if $y_k(\rho),z_k(\rho) \leq m_k$, then $y_{k+1}(\rho),z_{k+1}(\rho) \leq m_{k+1}$. Since $y_0(\rho) = z_0(\rho) = 1 = m_0$, we obtain $y_{k+1}(\rho),z_{k+1}(\rho) \leq m_{k+1}$ for all $k \in \{0,\ldots,\ell\}$, hence also $y_{\ell+1}(\rho),z_{\ell+1}(\rho) < 2$. Now, if $A < \frac1{17}$, $B < \frac1{100}$, then (iv) implies $x_{\ell+1}(\rho) < 1 + (A y_\ell(\rho)^2 + B z_\ell(\rho)^2 x_\ell(\rho)) < 1 + (4A + 4\frac75 B) < \frac75$, establishing the bootstrap assumption for $x_{\ell+1}$.
    
    It remains to show that the theorem hypotheses imply $a_s(\rho) < 1/17$ and $b_s(\rho) < 1/100$. That $b_s(\rho) \leq \frac1{100}$ follows immediately from \Cref{eq:abstractConditionB} since $\frac12 2^{|\alpha|} \geq |\alpha|$ for $|\alpha| \geq 1$. Similarly,
    \begin{align*}
        \sum_{\alpha \leq s} \frac{|\alpha|}{\alpha!} \rho^{|\alpha|} \| \partial^\alpha \left( A^\ast B A \right) \|_{X \to X^\ast} \leq \frac12 \sum_{1 \leq \alpha \leq s} \frac{(2\rho)^{|\alpha|}}{\alpha!} \| \partial^\alpha (A^\ast B A) \|_{X \to X^\ast}.
    \end{align*}
    Taking the adjoint is a linear isometry from $\mathcal{L}(X,H)$ to $\mathcal{L}(H^\ast,X^\ast)$. From this and the Leibniz rule it follows that
    \begin{align*}
        \sum_{|\alpha| \leq s} \frac{(2\rho)^{|\alpha|}}{\alpha!} \| \partial^\alpha (A^\ast B A) \|_{X \to X^\ast} \leq \left( \sum_{|\beta| \leq s} \frac{(2\rho)^{|\beta|}}{\beta!} \| \partial^\beta B \|_{H \to H^\ast} \right) \left( \sum_{|\gamma| \leq s} \frac{(2\rho)^{|\gamma|}}{\gamma!} \| \partial^\gamma A \|_{X \to H} \right)^2.
    \end{align*}
    Since $\|A_0\|_{X \to H} = \|B_0\|_{H \to H^\ast} = 1$, we find that
    \begin{align*}
        \sum_{|\alpha| \leq s} \frac{(2\rho)^{|\alpha|}}{\alpha!} \| \partial^\alpha (A^\ast B A) \|_{X \to X^\ast} \leq \left(1 + \frac{\tau}{20(\lambda^2 + \tau)}\right)^2\left(1 + \frac{\tau}{50(\lambda^2 + \tau)}\right) < 1 + \frac{2\tau}{17(\lambda^2 + \tau)}.
    \end{align*}
    In the last estimate, we used that $\frac{\tau}{\lambda^2+\tau} \leq 1$. Subtracting $1 = \|A_0^\ast B_0 A_0 \|_{X \to X^\ast}$, we find that $a_s(\rho) \leq \frac{1}{17}$ as desired.
\end{proof}

\subsection{Concrete estimates.} We specify the setup from \Cref{lem:abstractEstimates} to the clamped plate problem as follows. Let $D \subseteq \mathbb R^2$ be a smoothly bounded domain. We endow the spaces $H^2_{00}(D)$ and $H^1_{0}(D; \mathbb R^2)$ with norms given by
\begin{align*}
    \|u\|_{H^2_{00}(D)}^2 &= \sum_{j,k = 1}^2 \|\partial_j\partial_k u \|^2_{L^2(D)}, \\
    \|u\|_{H^1_{0}(D)}^2 &= \sum_{j = 1}^2 \|\partial_j u \|^2_{L^2(D)}, \\
    \|v\|_{H^1_{0}(D;\mathbb R^{2})}^2 &= \sum_{k = 1}^2 \|v_k \|^2_{H^1_0(D)}
\end{align*}
To apply \Cref{lem:abstractEstimates} to a family of Laplace operators on $D$, we require derivative estimates for $\Delta_{g_t}$, where $g_t$ is a parametric family of continously differentiable metrics on $D$. Let $\mathbb R_{\mathrm{sym}}^{2\times 2}$ denote the space of symmetric $2\times 2$ matrices. We endow the spaces $C^s(D)$ and $C^s(D,\mathbb R_{\mathrm{sym}}^{2\times 2})$ with the norms
\begin{align*}
    \|f\|_{C^1(D)} &= \|f\|_{C^0(D)} + \sum_{j=1}^2 \|\partial_j f\|_{C^0(D)}, \\
    \|h\|_{C^1(D, \mathbb R^{2\times 2}_{\mathrm{sym}})} &= \||h|_{2\to2}\|_{C^0(D)} + \sum_{j=1}^2 \||\partial_j h|_{2\to2}\|_{C^0(D)}.
\end{align*}
Here, $|h|_{2\to2}$ denotes the operator norm on $\mathbb R^2$ arising from the standard Euclidean norm.

\begin{lemma}
\label{lem:derivativeEstimatesForLaplacian}
Let $D \subseteq \mathbb R^2$ be a smoothly bounded domain contained in the disk of radius $2$.
To each $g \in C^1(D;\mathbb R_{\mathrm{sym}}^{2\times 2})$ with $\det(g)$ uniformly bounded away from zero, we associate the bilinear form $B_g(u,v) = \int_{D} \det(g)^{\frac12} uv$ and the differential operator 
    \begin{align*}
        \Delta_g u = \det(g)^{-\frac12} \divg \left( \det(g)^{\frac12} g^{-1} \cdot \nabla u \right)
    \end{align*}
Let $g_{(\cdot)} \in C^s((-1,1)^m,C^1(D;\mathbb R_{\mathrm{sym}}^{2\times 2}))$. Assume that $\rho > 0$ and $\delta \in (0,\frac{1}{100})$ are such that at $t=0$,
\begin{align*}
    \sum_{|\alpha| \leq s} \frac{\rho^{|\alpha|}}{\alpha!} \|\partial_t^\alpha (g_t - \mathds 1)\|_{C^1(D, \mathbb R^{2\times 2})} < \delta
\end{align*}
Then the estimates
    \begin{align}
        \label{eq:LipschitzConstantEstimateLaplacian}
        \sum_{|\alpha| \leq s} \frac{\rho^{|\alpha|}}{\alpha!} \| \partial_t^\alpha( \Delta_{g_t} - \Delta_{\mathds 1} )  \|_{H^2_{00}(D) \to L^2(D)} &\leq 6.23 \delta. \\
        \label{eq:LipschitzConstantEstimateVolumeForm}
        \sum_{|\alpha| \leq s} \frac{\rho^{|\alpha|}}{\alpha!} \| \partial_t^\alpha( B_{g_t} - B_{\mathds 1} )\|_{L^2(D) \to L^2(D)^\ast} &\leq 1.04 \delta
    \end{align}
    hold at $t=0$.
\end{lemma}

\begin{proof}
    Let $X$ be a Banach space and $f_{(\cdot)} \in C^s((-1,1)^m,X)$. Let us introduce the notation
    \begin{align}
        \mathcal S_X(f;s,\rho) := \sum_{|\alpha| \leq s} \frac{\rho^{|\alpha|}}{\alpha!} \|\partial_t^\alpha|_{t=0} f_t\|_X.
    \end{align}
    If $a:X \to Y$ is linear, then $\mathcal S_Y(af;s,\rho) \leq \|a\| \mathcal S_X(g;s,\rho)$.
    This definition also interacts nicely with the Leibniz rule: Given $f \in C((-1,1)^m,X)$, $g \in C((-1,1)^m,Y)$ and a bounded bilinear map $b: X \times Y \to Z$ with norm $\|b\|$, then 
    $\mathcal S_Z(b(f,g);s,\rho) \leq \|b\| \mathcal S_X(f;s,\rho) \mathcal S_Y(g;s,\rho)$. 
    
    The following estimates and identities follow directly from this property, elementary calculations, the (non-sharp) Poincar\'e inequality $\|u\|_{L^2} \leq \|u\|_{H^1_0}$ and the definitions above:
    \begin{enumerate}
        \item $\|\nabla u\|_{H^1_{0}} = \| u\|_{H^2_{00}}$,
        \item $\mathcal S_{C^1}( \det(g) g^{-1} - \mathds 1 ; s, \rho) = \mathcal S_{C^1}( g - \mathds 1 ; s, \rho)$,
        \item $\mathcal S_{C^1}( \det(g) - 1; s, \rho) \leq 2 \mathcal S_{C^1}( g - \mathds 1 ; s, \rho) (1 + \mathcal S_{C^1}( g - \mathds 1 ; s, \rho))$,
        \item $\mathcal S_{C^1}( f^{-\frac12} - 1; s, \rho) \leq \frac12 \mathcal S_{C^1}( f - 1 ; s, \rho) (1 - \mathcal S_{C^1}( f - 1 ; s, \rho))^{-1}$ for any $f \in C^1((-1,1)^m,C^1(D))$ with $\mathcal S_{C^1}( f - 1 ; s, \rho) < 1$,
        \item \label{item:PoincareInequality}$\|u\|_{L^2} \leq \|u\|_{H^1_0}$,
        \item $\|hv\|_{H^1_{0}} \leq \sqrt{2} \|h\|_{C^1} \|v\|_{H^1_0}$,
        \item $\|fu\|_{H^1_{0}} \leq \sqrt{2} \|f\|_{C^1} \|u\|_{H^1_0}$,
        \item $\| \divg v \|_{L^2} \leq \sqrt{2}\|v\|_{H^1_0}$.
    \end{enumerate}
    A brief justification of each of the claims above follows.
    \begin{enumerate}
        \item This follows directly from the definition:
        \begin{align*}
            \|\nabla u\|_{H^1_{0}}^2 = \sum_{j=1}^2 \|\partial_j u \|_{H^1_{0}}^2 = \sum_{j,k=1}^2 \|\partial_j \partial_k u \|_{L^2}^2 = \| u\|_{H^2_{00}}^2.
        \end{align*}
        \item In two dimensions, the matrices $\det(g) g^{-1} - \mathds 1$ and $g^T - \mathds 1$ are conjugate via the isometry $(x,y) \mapsto (-y,x)$. Since $g = g^T$, and the isometry does not depend on $t$,
        \begin{align*}
            | \partial_t^\alpha \left( \det(g) g^{-1} - \mathds 1 \right)|_{2 \to 2} = | \partial_t^\alpha \left( g - \mathds 1 \right)|_{2 \to 2}
        \end{align*}
        for all $\alpha \in \mathbb N^2$.
        \item The linear functional $a_{j,k} : C^1(D,\mathbb R^{2\times 2}_{\mathrm{sym}}) \to C^1(D)$, $a_{j,k} h = h_{jk}$ has norm $1$, since $a_{j,k} h = e_j^T h e_k$. Hence,
        $\mathcal S_{C^1}( g_{jk} - \delta_{jk}; s, \rho) \leq \mathcal S_{C^1}( g - \mathds 1; s, \rho)$ for all $j,k \in \{1,2\}$. By the triangle inequality and the product rule,
        \begin{align*}
            \mathcal S_{C^1}(\det g - 1; s, \rho) &= \mathcal S_{C^1}(g_{11} g_{22} - g_{12}^2 - 1; s, \rho) \\
            &\leq \mathcal S_{C^1}(g_{11} - 1; s, \rho) + \mathcal S_{C^1}(g_{22} - 1; s, \rho) \\ &+ \mathcal S_{C^1}(g_{11} - 1; s, \rho)\mathcal S_{C^1}(g_{22} - 1; s, \rho) + \mathcal S_{C^1}(g_{12}; s, \rho)^2 \\
            &\leq 2 \mathcal S_{C^1}(g - \mathds 1; s, \rho) (1 + \mathcal S_{C^1}(g - \mathds 1; s, \rho)).
        \end{align*}
        \item Since $C^1(D)$ is a Banach algebra, $\|f_0 - 1\|_{C^1(D)} \leq S_{C^1}(f - 1; s, \rho) < 1$ and the radius of convergence of the power series $(1+z)^{\frac12} = 1 + \sum_{j=1}^\infty \binom{\frac12}{j} z^j$ is $1$, the identity $f_t^{\frac12} - 1 = \sum_{j=1}^\infty \binom{\frac12}{j} (f-1)^j$ holds for all $t$ in a neighborhood of zero. Note that $\left|\binom{\frac12}{j}\right| \leq \frac12$ for all $j \geq 1$. The product rule and triangle inequality for $\mathcal S$ then imply
        \begin{align*}
            \mathcal S_{C^1}(f^{\frac12} - 1; s, \rho) &\leq \sum_{j=1}^\infty \left|\binom{\frac12}{j}\right| \mathcal S_{C^1}(f - 1; s, \rho)^j \\
            &\leq \frac12 \mathcal S_{C^1}(f - 1; s, \rho) \sum_{j=0}^\infty \mathcal S_{C^1}(f - 1; s, \rho)^j \\
            &\leq \frac12 \mathcal S_{C^1}(f - 1; s, \rho) (1 - \mathcal S_{C^1}(f - 1; s, \rho))^{-1}.
        \end{align*}
        \item Let $j_{0,1} = 2.40\dots$ denote the first nonnegative zero of the Bessel function $J_0$. Then,
        for all $u \in H^1_0(\mathbb D_{j_{0,1}})$, the inequality $\|u\|_{H^1_0} \geq \|u\|_{L^2}$ holds, since the lowest eigenvalue of the Dirichlet Laplacian on $\mathbb D_{j_{0,1}}$ equals $1$.
        Since our domain $D$ is assumed to be contained in a disk of radius $2 < j_{0,1}$, $H^1_0(D) \hookrightarrow H^1_0(\mathbb D_{j_{0,1}})$, and the inequality follows.
        \item Using \cref{item:PoincareInequality},
        \begin{align*}
            \| h u \|_{H^1_0}^2 &= \sum_{j=1}^2 \|\partial_j (hu) \|_{L^2}^2 = \sum_{j=1}^2 \| |\partial_j h|_{2\to 2} |u|_2 + |h|_{2 \to 2} |\partial_j u|_2\|_{L^2}^2 \\
            &\leq 2 \sum_{j=1}^2 (\| |\partial_j h|_{2 \to 2} \|_{C^0}^2 \|u\|_{L^2}^2 + \| |h|_{2 \to 2} \|_{C^0}^2 \|\partial_j u\|_{L^2}^2) \\
            &= 2 \|u\|_{L^2}^2 \sum_{j=1}^2 \||\partial_j h|_{2 \to 2} \|_{C^0}^2 + 2 \| |h|_{2\to 2} \|_{C^0}^2 \|u\|_{H^1_0}^2 \\
            &\leq 2 \|u\|_{H^1_0}^2 \Big( \||h|_{2 \to 2}\|_{C^0} + \sum_{j=1}^2 \||\partial_j h|_{2\to 2}\|_{C^0} \Big)^2 = 2 \|u\|_{H^1_0}^2 \|h\|_{C^1}^2.
        \end{align*}
        \item This follows from specializing $h$ to $f \mathds 1$ in item (6).
        \item By quadratic mean--arithmetic mean inequality, $\|\divg v \|_{L^2}^2 = \int_D (\partial_1 v_1 + \partial_2 v_2)^2 \leq 2 \int_D (|\partial_1 v_1|^2 + |\partial_2 v_2|^2) \leq 2 \|v\|_{H^1_0}^2$.
    \end{enumerate}
    If $\mathcal S_{C^1}( g-\mathds 1; s, \rho) \leq \delta$, then 
    \begin{align*}
        \mathcal S_{C^1}( \det(g) - 1; s, \rho) \leq 2 \mathcal S_{C^1}( g-\mathds 1; s, \rho) (1 + \delta) \leq 2 \delta (1+\delta),
    \end{align*}
    and hence
    \begin{align*}
        \mathcal S_{C^1}( (\det g)^{-\frac12} - 1; s, \rho) &\leq 
        \mathcal S_{C^1}( g-\mathds 1; s, \rho)(1+\delta)(1-2\delta-2\delta^2)^{-1} \\
        \mathcal S_{C^1}( (\det g)^{\frac12}g^{-1} - \mathds 1; s, \rho) &\leq 
        \mathcal S_{C^1}( (\det g)^{-\frac12} - 1; s, \rho)
        \mathcal S_{C^1}( (\det g) g^{-1}; s, \rho) \\ &+ 
        \mathcal S_{C^1}( (\det g) g^{-1} - \mathds 1; s, \rho) \\
        &\leq \mathcal S_{C^1}( (\det g)^{-\frac12} - 1; s, \rho) (1 + \mathcal S_{C^1}( g - \mathds 1; s, \rho)) + \mathcal S_{C^1}( g - \mathds 1; s, \rho) \\
        &\leq \mathcal S_{C^1}( g - \mathds 1; s, \rho) \big((1+\delta)^2(1-2\delta-2\delta^2)^{-1} + 1\big)
    \end{align*}
    Assuming $\delta = \frac1{100}$, the expressions involving $\delta$ can be estimated away, leading to 
    \begin{align*}
        \mathcal S_{C^1}( (\det g)^{-\frac12} - 1; s, \rho) &\leq 1.04 \mathcal S_{C^1}( g - \mathds 1; s, \rho), \\
        \mathcal S_{C^1}( (\det g)^{\frac12}g^{-1} - \mathds 1; s, \rho) &\leq 2.05 \mathcal S_{C^1}( g - \mathds 1; s, \rho).
    \end{align*}
    Writing $\Delta_g - \Delta_{\mathds{1}}$ as a telescoping sum,
    \begin{align*}
        \Delta_g - \Delta_{\mathds{1}} =
        (\det(g)^{-\frac12} - 1) \circ \divg \circ (\det(g)^{\frac12} g^{-1}) \circ \nabla + 
        \divg \circ (\det(g)^{\frac12} g^{-1} - \mathds 1) \circ \nabla,
    \end{align*}
    and applying 1---8 yields
    \begin{align*}
        \mathcal S_{\mathcal L(H^1_{00},L^2)}( \Delta_{g_t} - \Delta_{\mathds 1}; s, \rho) &\leq \mathcal S_{C^1}( \det(g)^{-\frac12} - 1 ; s, \rho)
        \cdot \sqrt{2} \cdot \sqrt{2} \cdot \mathcal S_{C^1}( \det(g)^{\frac12}g^{-1}; s, \rho) \\
        &+ \sqrt{2} \cdot \sqrt{2} \cdot \mathcal S_{C^1}( \det(g)^{\frac12}g^{-1} - \mathds 1; s, \rho) \\
        &\leq 2 \cdot 1.04 \cdot \mathcal S_{C^1}( g - \mathds 1; s, \rho) (1 + 2.05 \mathcal S_{C^1}( g - \mathds 1; s, \rho)) \\
        &+ 2 \cdot 2.05 \cdot \mathcal S_{C^1}( g - \mathds 1; s, \rho)
    \end{align*}
    Under the assumption $\mathcal S_{C^1}( g - \mathds 1; s, \rho) \leq \frac{1}{100}$ this reduces to 
    \begin{align*}
        \mathcal S_{\mathcal L(H^1_{00},L^2)}( \Delta_{g_t} - \Delta_{\mathds 1}; s, \rho) \leq 6.23 \mathcal S_{C^1}( g - \mathds 1; s, \rho).
    \end{align*}

    Now we estimate $B_g - B_{\mathds 1}$. Since $B_g = B_{\mathds 1} \circ \det(g)^{\frac12}$,
    \begin{align*}
        \mathcal S_{\mathcal L(L^2,(L^2)^\ast)}( B_{g_t} - B_{\mathds 1}; s, \rho)
        &\leq \mathcal S_{C^0}( \det(g)^{\frac12} - 1; s, \rho) \\
        &\leq \frac12 S_{C^0}( \det(g) - 1; s, \rho) (1 - S_{C^0}( \det(g) - 1; s, \rho))^{-1} \\
        &\leq S_{C^0}( g - \mathds 1; s, \rho) (1+\delta)(1 - 2\delta - 2\delta^2)^{-1}.
    \end{align*}
    Assuming $\delta \leq \frac{1}{100}$, $\mathcal S_{\mathcal L(L^2,(L^2)^\ast)}( B_{g_t} - B_{\mathds 1}; s, \rho) \leq 1.04 S_{C^0}( g - \mathds 1; s, \rho)$.
\end{proof}

\begin{lemma}
    \label{lem:concreteEstimates}
    Let $D \subseteq \mathbb R^2$ be a smoothly bounded domain contained in a disk of radius $2$. Let $X_1,\ldots,X_m \in C^\infty(\overline{D},\mathbb R^2)$ be $D_N$-equivariant vector fields. Assume that $X_1,\ldots,X_m \equiv 0$ on a subdomain $K \subseteq D$ with $\overline{K} \subseteq D$.
	Let $M = \sum_{j=1}^m \|J_{X_j} \|_{C^1(D,\mathbb R^{2\times 2})}$.
    
    Define $\phi_{(\cdot)}:(-1,1)^m \times \overline{D} \to \mathbb R^2$ by $\phi_t(x) := x + \sum_{j=1}^m t_j X_j(x)$ and write $D_t = \phi_t(D)$. Given $\rho > 0$, write $\mathcal D_\rho = \{(t,x): x \in \phi_t(D), t \in (-\rho,\rho)^m \}$.
    
	Let $\lambda_0$ be an eigenvalue that is simple in the $D_N$-invariant spectrum of $\Delta^2$ on $D_0$. Let $u_0$ be a corresponding eigenfunction with $\int_{D_0} u_0^2 = 1$. Let $\tau > 0$ be such that $(\lambda_0^2 - \tau, \lambda_0^2 + \tau)$ contains no points in the $D_N$-invariant spectrum of $\Delta^2$ besides $\lambda_0^2$.
    
	Let $\rho = \frac{\tau}{2400 M (\lambda_0^2 + \tau)}$. Then there exists a unique $D_N$-invariant function $u \in C^\infty(\overline{\mathcal D_\rho})$ with $\Delta^2 u_t = \lambda_t^2 u_t$, $\int_{D_t} u_t^2 = 1$ and $u_t|_{t=0} = u_0$. The following hold for all $t \in (-\rho,\rho)^m$:
	\begin{align}
        \label{eq:concreteDerivativeEstimates}
        \begin{split}
		\sum_{|\alpha| \in \mathbb N^m} \frac{\rho^{|\alpha|}}{\alpha!} |\partial_t^\alpha \lambda_t| &= 2 \tau, \\ 
		\sum_{|\alpha| \in \mathbb N^m} \frac{\rho^{|\alpha|}}{\alpha!} \| \partial_t^\alpha u \|_{H^2(K)} &\leq 2(\lambda_0^2 + \tau)^{\frac12}, \\
		\sum_{|\alpha| \in \mathbb N^m} \frac{\rho^{|\alpha|}}{\alpha!} \| \partial_t^\alpha u \|_{L^2(K)} &\leq 2.
        \end{split}
	\end{align}
\end{lemma}

\begin{proof}
    Denote $g_t = \phi_t^\ast g_{\mathbb R^2}$. The Laplace operator $\Delta_{g_{\mathbb R^2}}: H^2_{00}(D_t) \to L^2(D_t)$ and the inner product $B_{g_{\mathbb R^2}}: L^2(D_t) \to L^2(D_t)^\ast$ are conjugate to $\Delta_{g_t} : H^2_{00}(D) \to L^2(D)$ and $B_{g_t}: L^2(D) \to L^2(D)^\ast$ via the pull-back $\phi_t^\ast$. Note that $g_t$ is analytic in $t$ (in fact, quadratic), hence $\Delta_{g_{(\cdot)}}$ is a real-analytic family of operators in $\mathcal L(H^2_{00}(D), L^2(D))$. It follows from basic perturbation theory that there exists $\varepsilon > 0$ and analytic functions $\tilde u_{(\cdot)}:(-\varepsilon,\varepsilon)^m \to H^2_{00}(D)$, $\lambda_{(\cdot)}:(-\varepsilon,\varepsilon)^m \to \mathbb R$ such that $\tilde u_t|_{t=0} = u_0$, $\Delta_{g_t}^2 \tilde u_t = \lambda_t^2 \tilde u_t^2$ and $\int_{D} \tilde u_t^2 \det(g_t)^{\frac12} = 1$. By a bootstrap argument, $\varepsilon$ may be enlarged until $\inf_{t \in (-\varepsilon,\varepsilon)^m} \mathrm{dist}(\lambda_t^2,\sigma(\Delta_{g_t}^2) \setminus \{\lambda_t^2\}) = 0$.
    
    The function $u_t \in H^2_{00}(D_t)$ is then given by $(\phi_t^{-1})^\ast \tilde u_t$. Since $X_1,\ldots,X_m$ are assumed to vanish on $K \subseteq D$, $u_t$ agrees with $\tilde u_t$ on $K$, and any derivative estimates on $\tilde u_t$ directly translate to estimates for $u_t$ on $K$. The remainder of the proof is concerned with obtaining explicit bounds on $\varepsilon$ and on the derivatives of $\lambda_t$ and $\tilde u_t$ (seen as a function valued in $H^2_{00}(D)$ or $L^2(D)$).
    
    Define $\Phi^t_{t'}: D_t \to \mathbb R^2$, 
    \begin{align*}
        \Phi^t_{t'}(x) = x + \sum_{j=1}^m t'_j \, (X_j \circ \phi_t^{-1})(x),
    \end{align*}
    and write $g^t_{t'} = (\Phi^t_{t'})^\ast g_{\mathbb R^2} = J_{\Phi^t_{t'}}^T J_{\Phi^t_{t'}}$. The maps $\Phi^t_{t'}$ enjoy the semigroup-like property $\Phi^t_{t'} \circ \phi_t = \phi_{t+t'}$. Thus, $\Delta: H^2_{00}(D_{t+{t'}}) \to L^2(D_{t+{t'}})$ and $B_{g_{\mathbb R^2}}: L^2(D_{t+t'}) \to L^2(D_{t+t'})^\ast$ are conjugate to $\Delta_{g^t_{t'}}: H^2_{00}(D_{t}) \to L^2(D_{t})$ and $B_{g^t_{t'}}: L^2(D_{t}) \to L^2(D_{t})^\ast$ via $(\Phi^t_{t'})^\ast$. By the Leibniz rule,
    \begin{align*}
        \mathcal S_{C^1(D_t,\mathbb R^{2\times 2})}(g^t_{(\cdot)} - \mathds 1,s,\rho) &= \mathcal S_{C^1(D_t,\mathbb R^{2\times 2})}(J_{\Phi^t_{(\cdot)}}^T J_{\Phi^t_{(\cdot)}} - \mathds 1,s,\rho) \\ &\leq \mathcal S_{C^1(D_t,\mathbb R^{2\times 2})}(J_{\Phi^t_{(\cdot)}} - \mathds 1,s,\rho)(2 + \mathcal S_{C^1(D_t,\mathbb R^{2\times 2})}(J_{\Phi^t_{(\cdot)}} - \mathds 1,s,\rho)),
    \end{align*}
    and since $J_{\Phi^t_{t'}}$ is linear in $t'$,
    \begin{align*}
        \mathcal S_{C^1(D_t,\mathbb R^{2\times 2})}(J_{\Phi^t_{(\cdot)}} - \mathds 1,s,\rho) = \rho \sum_{j=1}^m \| J_{X_j \circ \phi_t^{-1}} \|_{C^1(D_t,\mathbb R^{2\times 2})}
    \end{align*}
    In particular, $\mathcal S_{C^1(D,\mathbb R^{2\times 2})}(J_{\phi_t} - \mathds 1,s,\rho) = M$.

    We wish to apply \Cref{lem:abstractEstimates} to the operators $A_{t'} = \Delta_{g^t_{t'}}$ and $B_{t'} = B_{g^t_{t'}}$ operating on the spaces $H^2_{00}(D_t)$ and $L^2(D_t)$. The required estimates on the derivatives of $A_{t'}$ and $B_{t'}$ will follow from \Cref{lem:derivativeEstimatesForLaplacian}. To apply it, we first need derivative estimates for $g^{t}_{(\cdot)}: (-1,1)^m \to C^1(D_t, \mathbb R^{2\times 2})$ at $t' = t$. Let $X \in \{X_1,\ldots,X_m\}$ and consider $X \circ \phi_t^{-1}$. By the chain rule,
	\begin{align*}
		J_{X \circ \phi_t^{-1}} = (J_X \circ \phi_t^{-1}) \cdot J_{\phi_t^{-1}}.
	\end{align*}
    By the Leibniz rule,
    \begin{align*}
    \left\|  J_{X \circ \phi_t^{-1}} \right\|_{C^1(D_t,\mathbb R^{2\times 2})} \leq \left\| J_X \circ \phi_t^{-1} \right\|_{C^1(D_t,\mathbb R^{2\times 2})}
    \left\| J_{\phi_t^{-1}} \right\|_{C^1(D_t,\mathbb R^{2\times 2})}
    \end{align*}
    First, we consider $\| J_{\phi_t^{-1}}\|_{C^1(D_t,\mathbb R^{2\times 2})}$. A power series argument using submultiplicativity of $\| \cdot \|_{C^s}$, $s \in \{0,1\}$, shows that
    \begin{align*}
        \| J_{\phi_t}^{-1} - \mathds 1 \|_{C^s(D_t,\mathbb R^{2\times 2})} \leq (1 - \| J_{\phi_t} - \mathds 1 \|_{C^s(D_t,\mathbb R^{2\times 2})})^{-1} - 1\leq (1-|t|_\infty M)^{-1} - 1
    \end{align*}
    if $s \in \{0,1\}$ and $|t|_\infty M < 1$. Since $J_{\phi_t^{-1}} \circ \phi_t = J_{\phi_t}^{-1}$ and composition does not change sup-norms, $\|J_{\phi_t^{-1}}\|_{C^0(D_t,\mathbb R^{2\times 2})} \leq (1-|t|_\infty M)^{-1}$.
    Let $F \in C^1(D,\mathbb R^{2\times 2})$. Then 
    \begin{align*}
        | \partial_j (F \circ \phi_t^{-1}) |_{2 \to 2}
        &\leq \sum_{\ell=1}^2 |\partial_\ell F \circ \phi_t^{-1}|_{2 \to 2} \left|[J_{\phi_t^{-1}}]^\ell_j \right| \\
		\| F \circ \phi_t^{-1} \|_{C^1(D_t,\mathbb R^{2\times 2})} &\leq \| F \circ \phi_t^{-1} \|_{C^0(D_t,\mathbb R^{2\times 2})} + 2 \sup_{\ell,j} \left\|[J_{\phi_t^{-1}}]^\ell_j \right\|_{C^0(D_t)} \sum_{\ell=1}^2 \|\partial_\ell F\|_{C^0(D)} \\ &\leq 2 (1-|t|_\infty M)^{-1} \| F \|_{C^1(D,\mathbb R^{2\times 2})}.
    \end{align*}
	It follows that
	\begin{align*}
		\| J_{X} \circ \phi_t^{-1} \|_{C^1(D_t,\mathbb R^{2\times 2})} &\leq 2 (1-|t|_\infty M)^{-1} \| J_X \|_{C^1(D,\mathbb R^{2\times 2})}, \\
		\| J_{\phi_t^{-1}} \|_{C^1(D_t,\mathbb R^{2\times 2})} &\leq 2 (1-|t|_\infty M)^{-1} \| J_{\phi_t}^{-1} \|_{C^1(D,\mathbb R^{2\times 2})} \leq 2 (1-|t|_\infty M)^{-2}
    \end{align*}
	the latter since $J_{\phi_t^{-1}} = J_{\phi_t}^{-1} \circ \phi_t^{-1}$. Thus,
	\begin{align*}
		\| J_{X_\ell \circ \phi_t^{-1}} \|_{C^1(D_t,\mathbb R^{2\times 2})} \leq 4 (1-|t|_\infty M)^{-3} \| J_X \|_{C^1(D,\mathbb R^{2\times 2})}
    \end{align*}
	Summing up over $X_1,\ldots,X_m$ yields
	\begin{align*}
		\sum_{\ell=1}^m \| J_{X_\ell \circ \phi_t^{-1}} \|_{C^1(D_t,\mathbb R^{2\times 2})} \leq 4 (1-|t|_\infty M)^{-3} \sum_{\ell=1}^m \| J_{X_\ell} \|_{C^1(D,\mathbb R^{2\times 2})} = 4 (1-|t|_\infty M)^{-3} M,
	\end{align*}
	and hence 
	\begin{align*}
		\mathcal S_{C^1(D_t,\mathbb R^{2\times 2})}(g^t_{(\cdot)} - \mathds 1,s,\rho) &\leq 4 M \rho \, (1-|t|_\infty M)^{-3} \left( 2 + 4 M \rho \, (1-|t|_\infty M)^{-3} \right)
	\end{align*}
	Assume $|t|_\infty, \rho \leq \frac{1}{1000 M}$. Then 
	\begin{align*}
		\mathcal S_{C^1(D_t,\mathbb R^{2\times 2})}(g^t_{(\cdot)} - \mathds 1,s,\rho) &\leq 8.05 M \rho
	\end{align*}
	Since $\rho \leq \frac{1}{1000 M}$, $\mathcal S_{C^1(D_t,\mathbb R^{2\times 2})}(g^t_{(\cdot)} - \mathds 1,s,\rho) \leq \frac1{100}$, hence we may apply \Cref{lem:derivativeEstimatesForLaplacian} to obtain
	\begin{align}
        \sum_{|\alpha| \leq s} \frac{\rho^{|\alpha|}}{\alpha!} \| \partial_t^\alpha( \Delta_{g_t} - \Delta_{\mathds 1} )  \|_{H^2_{00}(D) \to L^2(D)} &\leq 51 M \rho. \\
        \sum_{|\alpha| \leq s} \frac{\rho^{|\alpha|}}{\alpha!} \| \partial_t^\alpha( B_{g_t} - B_{\mathds 1} )\|_{L^2(D) \to L^2(D)^\ast} &\leq 9 M \rho
	\end{align}
    In particular, $\sum_{j=1}^m \| \partial_{t_j} \Delta_{g_t} \|_{H^2_{00}(D) \to L^2(D)} \leq 51 M$ and $\sum_{j=1}^m \| \partial_{t_j} B_{g_t} \|_{L^2(D) \to L^2(D)^\ast} \leq 9 M$. Thus, any eigenvalue $\mu_t$ in the window $(\lambda_0^2 - \tau, \lambda_0^2 + \tau)$ is subject to the derivative bound 
    \begin{align}
        \label{eq:concreteFirstDerivativeBound}
        \sum_{j=1}^m \| \partial_{t_j} (\mu_t^2) \|_{L^2(D) \to L^2(D)^\ast} \leq 120(\lambda_0^2 + \tau)M
    \end{align}
    for almost every $t \in (-\rho,\rho)$ -- note that \eqref{eq:firstDerivativeEstimate} does not reference the spectral gap and holds for any eigenvalue, since its proof rests only on the Hellmann-Feynman formula, which applies to eigenvalues of higher multiplicity as well.
    
    Setting $\varepsilon = \frac{\tau}{2400 M (\lambda_0^2 + \tau)}$ guarantees that $\lambda_t^2$ stays at a distance of at least $\frac{9}{10}\tau$ from the remaining spectrum of $\Delta_{g_t}^2$ for all $t \in (-\varepsilon,\varepsilon)^m$: The window $(\lambda_0^2 - \tau,\lambda_0^2 + \tau)$ is initially free of eigenvalues other than $\lambda_t^2$, which is constrained by \eqref{eq:concreteFirstDerivativeBound} to remain within the smaller window $(\lambda_0^2-\frac1{20}\tau,\lambda_0^2+\frac1{20}\tau)$ (note that $\frac{120}{2400} = \frac{1}{20}$), and any other eigenvalue is constrained by the same estimate from entering $(\lambda_0^2-\frac{19}{20}\tau,\lambda_0^2+\frac{19}{20}\tau)$.
    
	Let $\rho = \frac{\tau}{2400 M (\lambda_0^2 + \tau)}$. The conditions for \Cref{lem:abstractEstimates} are then satisfied for any $t \in (-\varepsilon,\varepsilon)^m$, so we obtain
	\begin{align*}
		\mathcal S_{C^1}(\lambda_{(\cdot)}^2 - \lambda_t^2,s,\rho) &\leq 2 \tau, \\
        \mathcal S_{C^1((-\varepsilon,\varepsilon)^m,H^2_{00}(D_t))}(u_{(\cdot)},s,\rho) &\leq 2(\lambda_0^2 + \tau)^{\frac12} \\
        \mathcal S_{C^1((-\varepsilon,\varepsilon)^m,L^2(D_t))}(u_{(\cdot)},s,\rho) &\leq 2,
	\end{align*}
    which implies \eqref{eq:concreteDerivativeEstimates}, since $\|\cdot \|_{H^2(K)} \leq \|\cdot \|_{H^2_{00}(D_t)}$ and $\|\cdot \|_{L^2(K)} \leq \|\cdot \|_{L^2(D_t)}$.
\end{proof}

\section{Constructing a suitable one-parameter family of perturbations}

We choose three vector fields $X_1$, $X_2$ and $X_3$ to which we will apply \Cref{lem:concreteEstimates}. At the boundary of the unit disk, these vector fields will satisfy
\begin{align}
    X_1 \cdot n &= \cos(2N \theta) + 1, & X_2 \cdot n &= \cos(N \theta), & X_3 \cdot n &= \cos(3N \theta) + \cos(2N \theta) - \tfrac12.
\end{align}
At $t=0$, the derivatives of the quantities $\lambda_t$, $v_t(0)$ and $w_t(0)$ are then simple and explicit:
\begin{align}
    \label{eq:derivative_expressions1}
    \partial_{t_1} (\lambda_t^2) &= 6 \lambda^2, \\
    \label{eq:derivative_expressions2}
    \partial_{t_2} (\lambda_t^2) &= \partial_{t_3} (\lambda_t^2) = 0, \\
    \label{eq:derivative_expressions3}
    \partial_{t_2} v_t(0) &= \frac{\sqrt{\lambda}}{\sqrt{\pi}} J_0(\sqrt{\lambda})^{-1} \left( J_0(\sqrt{\lambda})^{-1} J_0'(\sqrt{\lambda}) - I_0(\sqrt{\lambda})^{-1} I_0'(\sqrt{\lambda}) \right)^{-1}, \\
    \label{eq:derivative_expressions4}
    \partial_{t_1} v_t(0) &= \partial_{t_3} v_t(0) = 0, \\
    \label{eq:derivative_expressions5}
    \partial_{t_3}^2 w_t(0) &= - \frac{\lambda}{\sqrt{\pi}} I_0(\sqrt{\lambda})^{-1}.
\end{align}
Here, \eqref{eq:derivative_expressions1} and \eqref{eq:derivative_expressions2} were derived from \eqref{eq:derivative_of_lambdasquared}, \eqref{eq:derivative_expressions3} and \eqref{eq:derivative_expressions4} from \eqref{eq:derivative_of_v0}, and \eqref{eq:derivative_expressions5} from \eqref{eq:secondDerivativeOfW}.
For technical reasons, we need $X_1,X_2,X_3$ to be supported outside a disk of radius $1-\frac1N$. To achieve this, we choose a bump function $\chi \in C^\infty(\mathbb R)$ with $\chi \equiv 0$ on $(-\infty,e^{-2})$ and $\chi \equiv 1$ on $(1,\infty)$. Then $\chi(|x|^{2N}) = 0$ for all $x$ with $|x| \leq 1-\frac1N$, since $(1-\frac1N)^{2N} \leq e^{-2}$ for all $N \in \mathbb N$. The function $\chi$ can be chosen to satisfy $\max(\|\chi\|_{C^0},\|\chi''\|_{C^0},\|\chi''\|_{C^0}) \leq 6$. Define the vector fields
\begin{align}
    \label{eq:vectorFieldDefinitions}
    \begin{split}
    X_1 &= \Re((x+iy)^{2N} + 1) \chi((x^2+y^2)^N)(x \partial_x + y \partial_y),\\
    X_2 &= \Re((x+iy)^N) \chi((x^2+y^2)^N) (x \partial_x + y \partial_y), \\
    X_3 &= \Re((x+iy)^{3N} + (x+iy)^{2N} - \tfrac12) \chi((x^2+y^2)^N)(x \partial_x + y \partial_y).
    \end{split}
\end{align}
Our next priority is to compute the $C^1$-norms of the Jacobians of $X_1$, $X_2$ and $X_3$, which enters \Cref{lem:concreteEstimates} in form of the number $M$ in that lemma's statement.

\begin{lemma}
    \label{lem:explicitNorms}
    Assume $N \geq 10$. With $X_1$, $X_2$ and $X_3$ defined as in \Cref{eq:vectorFieldDefinitions}, 
    \begin{align*}
        \|J_{X_1}\|_{C^1(\mathbb D;\mathbb R^{2\times 2})} &\leq 600 N^2, & \|J_{X_2}\|_{C^1(\mathbb D;\mathbb R^{2\times 2})} &\leq 230 N^2, & \|J_{X_3}\|_{C^1(\mathbb D;\mathbb R^{2\times 2})} &\leq 1070 N^2.
    \end{align*}
\end{lemma}

\begin{proof}
    To each $K \in \mathbb N$, we associate the expression
	\begin{align*}
		F_K = \Re((x+iy)^K) \chi((x^2+y^2)^N) (x \partial_x + y \partial_y)
	\end{align*}
	By the Leibniz rule,
	\begin{align*}
		\max_{j,k \in \{1,2\}} \| \partial_j F_K^k \|_{C^0} &\leq (K+2N+1) \max( \|\chi\|_{C^0},\|\chi'\|_{C^0} ), \\
		\max_{j,k,\ell \in \{1,2\}} \| \partial_j \partial_k F_K^\ell \|_{C^0} &\leq (K^2 + 4KN + 4N^2 + K + 2N) \max( \|\chi\|_{C^0},\|\chi'\|_{C^0},\|\chi''\|_{C^0}).
	\end{align*}
	Since any $d \times d$ matrix $A$ satisfies the inequality $\|A\|_{2 \to 2} \leq d \sup_{j,k \in \{1,2\}} |A_{jk}|$,
	\begin{align*}
		\| J_{F_K} \|_{C^0(\mathbb D; \mathbb R^{2\times 2})} &\leq 2 (K+2N+1) \max( \|\chi\|_{C^0},\|\chi'\|_{C^0} ), \\
		\max_{j \in \{1,2\}} \| \partial_j J_{F_K} \|_{C^0(\mathbb D; \mathbb R^{2\times 2})} &\leq 2 (K^2 + 4KN + 4N^2 + K + 2N) \max( \|\chi\|_{C^0},\|\chi'\|_{C^0},\|\chi''\|_{C^0}).
	\end{align*}
    Summing up, we obtain
    \begin{align*}
        \| J_{F_K} \|_{C^1(\mathbb D; \mathbb R^{2\times 2})} \leq (4K^2 + 16 KN + 16N^2 + 12N + 6K + 2)  \max( \|\chi\|_{C^0},\|\chi'\|_{C^0},\|\chi''\|_{C^0}).
    \end{align*}
    Since $X_1 = F_{2N} + F_0$, $X_2 = F_N$ and $X_3 = F_{3N} + F_{2N} - \frac12 F_{0}$,
    \begin{align*}
        \| J_{X_1} \|_{C^1(\mathbb D; \mathbb R^{2\times 2})} &\leq (80N^2 + 36N + 4) \max( \|\chi\|_{C^0},\|\chi'\|_{C^0},\|\chi''\|_{C^0}), \\
        \| J_{X_2} \|_{C^1(\mathbb D; \mathbb R^{2\times 2})} &\leq (36N^2 + 18N + 2) \max( \|\chi\|_{C^0},\|\chi'\|_{C^0},\|\chi''\|_{C^0}), \\
        \| J_{X_3} \|_{C^1(\mathbb D; \mathbb R^{2\times 2})} &\leq (172 N^2 + 60N + 5) \max( \|\chi\|_{C^0},\|\chi'\|_{C^0},\|\chi''\|_{C^0}).
    \end{align*}
    Under the assumptions $5 \leq N$ and $\max( \|\chi\|_{C^0},\|\chi'\|_{C^0},\|\chi''\|_{C^0}) \leq 6$, these expressions can be simplified, yielding the claimed bounds.
\end{proof}

Now let $\Phi_t(x) = x + t_1 X_1(x) + t_2 X_2(x) + t_3 X_3(x)$, write $D_t = \Phi_t(\mathbb D)$ and consider $u_t$, $\lambda_t$ as provided by \Cref{lem:concreteEstimates}.

\begin{lemma}
    \label{lem:oneParameterFamily}
    Assume $N \geq 10$, $\sqrt{\lambda_0} \in (N+N^{\frac13}, N + 3 N^{\frac13})$ and $\frac{\tau}{\lambda_0^2} \leq \frac1{10}$. Let $\varepsilon = 10^{-17} N^{-5}$. There exists a function $g:(-\varepsilon,\varepsilon) \to \mathbb R^2$ such that $\lambda_{(s,g(s))} = \lambda_0$ and such that $\tilde u_s := u_{(s,g(s))}$ satisfies $\frac{1}{2\lambda_0}(\Delta \tilde u_s - \lambda_0 \tilde u_s)(0) = 0$. Let $\tilde w_s = \tfrac{1}{2\lambda_0}\left(\Delta \tilde u_s + \lambda_0 \tilde u_s \right)$ denote the SP component of $\tilde u_s$. Then, for all $s \in (-\varepsilon,\varepsilon)$, 
    \begin{align*}
        \left|\partial_s^3 \tilde w_s(0) \right| \leq 600 \cdot 10^{40} N^{13.5} I_0(\sqrt{\lambda_0})^{-1}.
    \end{align*}
\end{lemma}

\begin{proof}
    The proof is a straightforward application of the implicit function theorem and the concrete derivative estimates of \Cref{lem:concreteEstimates}. Let $\rho = \frac{\tau}{2400 M (\lambda_0^2 + \tau)}$ denote the radius from \Cref{lem:concreteEstimates}, where $M = \sum_{\ell = 1}^3 \| J_{X_\ell} \|_{C^1(\mathbb D)}$. By \Cref{lem:explicitNorms}, $M \leq 1900 N^2$, so $\rho \geq 2 \cdot 10^{-7} N^{-2} \frac{\tau}{(\lambda_0^2 + \tau)}$. By \Cref{lem:nondegenerate_eigenvalues}, $\lambda_0^2 \leq (N+5N^{\frac13})^4 \leq 20 N^4$ and $\tau \geq 4N^3$, hence $\frac{\lambda_0^2+\tau}{\tau} \leq 5 N + 1$. If $N \geq 10$, $\frac{\lambda_0^2+\tau}\tau \leq 5.1 \cdot N$, yielding $\rho \geq 10^{-8} N^{-3}$.
    
    From \Cref{lem:concreteEstimates}, we obtain $\lambda_{(\cdot)}^2 : (- \rho, \rho)^3 \to \mathbb R$ and a parametric family $u_t \in H^2_{00}(D_t)$ of eigenfunctions together with explicit derivative estimates on $\lambda_t^2$ and $u_t|_{\mathbb D_{1-\frac1N}}$ (measured both in $L^2(\mathbb D_{1-\frac1N})$ and in $H^2(\mathbb D_{1-\frac1N})$). In the following, we write $\mathbb D' = \mathbb D_{1-\frac1N}$.
    Define
    \begin{align}
        R_1(t) &= \lambda_t^2 - \lambda_0^2, \\ 
        R_2(t) &= \| J_0(\sqrt{\lambda_0} (\cdot) ) \|_{L^2(\mathbb D')}^{-2} \int_{\mathbb D} J_0(\sqrt{\lambda_0} x) \frac{1}{2\lambda_0} \left( \Delta u_t - \lambda_0 u_t \right)(x) dx \\
        F(t) &= \| I_0(\sqrt{\lambda_0} (\cdot) ) \|_{L^2(\mathbb D')}^{-2} \int_{\mathbb D'} I_0(\sqrt{\lambda_0} x) \frac{1}{2\lambda_0} \left( \Delta u_t + \lambda_0 u_t \right)(x) dx
    \end{align}
    The functions $R_2$ and $F$ are defined in this way to simplify the subsequent calculations. They agree with $v_t(0)$ and $w_t(0)$, the Helmholtz and SP components' values at the origin, wherever $\lambda_t = \lambda_0$, but their derivatives are more easily estimated. Using \Cref{lem:concreteEstimates} and the observation that $\|\Delta u \|_{L^2(\mathbb D')} \leq \sqrt{2} \| u \|_{H^2(\mathbb D')}$,
    \begin{align}
        \label{eq:bound_on_eigenvalue_derivatives}
        \sum_{|\alpha| \leq s} \frac{\rho^\alpha}{\alpha!} |\partial_t^\alpha R_1(t)| &\leq 2\tau \\
        \label{eq:bound_on_Helmholtz_derivatives}
        \begin{split}
        \sum_{|\alpha| \leq s} \frac{\rho^\alpha}{\alpha!} |\partial_t^\alpha R_2(t)| &\leq \| J_0(\sqrt{\lambda_0}(\cdot))\|_{L^2(\mathbb D')}^{-1} \tfrac{1}{2\lambda_0} \left( 2 \sqrt{2} (\lambda_0^2 + \tau)^{\frac12} + 2 \lambda_0 \right) \\ & \leq 4 \| J_0(\sqrt{\lambda_0}(\cdot))\|_{L^2(\mathbb D')}^{-1}
        \end{split}\\
        \label{eq:bound_on_SP_derivatives}
        \begin{split}
        \sum_{|\alpha| \leq s} \frac{\rho^\alpha}{\alpha!} |\partial_t^\alpha F(t)| &\leq \| I_0(\sqrt{\lambda_0}(\cdot))\|_{L^2(\mathbb D')}^{-1} \tfrac{1}{2\lambda_0} \left( 2\sqrt{2} (\lambda_0^2 + \tau)^{\frac12} + 2 \lambda_0 \right) \\ &\leq 4 \| I_0(\sqrt{\lambda_0}(\cdot))\|_{L^2(\mathbb D')}^{-1}
        \end{split}
    \end{align}
    The final expressions in \eqref{eq:bound_on_Helmholtz_derivatives} and \eqref{eq:bound_on_SP_derivatives} were obtained using the assumption $\frac{\tau}{\lambda_0^2} \leq \frac1{10}$.
    
    The remainder of the proof is dedicated to obtaining explicit derivative estimates on the function $g$ obtained by the implicit function theorem as well as the function $F(s,g(s))$, which is just the SP component of $\tilde u(s)$ evaluated at the origin.
    
    To simplify presentation, the situation at hand shall be abstracted as follows: Consider a function $R(x,y): (-\rho,\rho)^{1 + n} \to \mathbb R^n$ whose components satisfy the estimates,
    \begin{align*}
        \sum_{\alpha \leq 3} \frac{\rho^{|\alpha|}}{\alpha!} |\partial^\alpha R_k| \leq E_k, \ k = 1,\ldots,n.
    \end{align*}
    Assume furthermore that $\partial_x R(0,0) = 0$ and $A := D_y R(0,0) = \diag(\mu_1,\ldots,\mu_n)$ for some $\mu_1,\ldots,\mu_n \in \mathbb R \setminus \{0\}$.
    Let $\kappa = \max_k(|\mu_k|^{-1} E_k) \rho^{-1}$. Consider $F(x,y) = y - A^{-1} R(x,y)$, the fixed point iteration from the proof of the implicit function theorem. To simplify notation, we will employ Einstein's summation convention: All indices occurring twice are tacitly summed over. Then
    \begin{align*}
        F(x,y)  &= y - A^{-1} R(x,y) \\
                &= \left( \int_0^1 (\mathds 1 - A^{-1} D_y R(x,t y)) dt \right) y - A^{-1} R(x,0) \\
                &= \left( \int_0^1 A^{-1} (A - D_y R(x,t y)) dt \right) y - \left( \int_0^1 A^{-1} \partial_x R(t x,0) dt \right) x \\
                &= A^{-1} \left( \left( \int_0^1 (A - D_y R(x,t y)) dt \right) y - \left( \int_0^1 \partial_x R(t x,0) dt \right) x \right)\\
                &= - A^{-1} \Bigg( \left( \int_0^1 \int_0^1 \partial_{x} \partial_{y_k} R(s x,s t y)) ds dt \right) x y_k 
                \\ &- \left( \int_0^1 \int_0^1 \partial_{y_j} \partial_{y_k} R(s x,s t y)) t ds dt \right) y_j y_k \\ &- \left( \int_0^1 \int_0^1 \partial_{x}^2 R(s t x,0) t ds dt \right) x^2 \Bigg)
    \end{align*}
    If $(x,y) \in (-\sigma \rho,\sigma \rho)^{1+n}$, $\sigma \in (0,1)$, then the triangle inequality applied to the above implies $|F(x,y)_k| \leq 2 \sigma^2 \mu_k^{-1} E_k$.
    To close the fixed point iteration, $\sigma$ must be small enough so $(x,y) \in (-\sigma \rho,\sigma \rho)^{1+n}$ implies $F(x,y) \in (-\sigma \rho,\sigma \rho)^{1+n}$. It suffices to set $\sigma \leq \frac c2 \rho \min_k( \mu_k E_k^{-1} )$, i.e.\ $\sigma \leq \frac{c}{2 \kappa}$. Here, $c \in (0,1)$ is arbitrary, but it will be convenient later to choose $c = \frac12$.
    
    This choice also ensures $F(x,\cdot)$ is a contraction. To show this, we first estimate the deviation of $D_y R(x,y)$ from $A$ in an appropriate norm. The most convenient norm turns out to be the $|\cdot|_\infty$-norm and the matrix norm derived from it, both of which we will simply denote by $|\cdot|$ in this proof.
    \begin{align*}
        |A^{-1}(A - D_y R(x,y))| &= \left| \int_0^1 A^{-1} \left( \partial_{x} D_{y} R(s x,s y)) x + \partial_{y_j} D_{y} R(s x,s y)) y_j \right) ds \right| \\
        &= \max_{k=1,\ldots,n} \sum_{\ell=1}^m \left| \int_0^1 A^{-1} \left( \partial_{x} \partial_{y_\ell} R_k (s x,s y)) x + \partial_{y_j} \partial_{y_\ell} R_k(s x,s y)) y_j \right) ds \right| \\
        &\leq 2 \max_{k}(\mu_k^{-1} E_k) \rho^{-2} \sigma \rho \leq 2 \kappa \sigma \leq c.
    \end{align*}
     In the following calculation, $y_t$ is shorthand for $y_0 + t(y_1 - y_0)$:
    \begin{align*}
        |F(x,y_1) - F(x,y_0)|  &= |y_1 - y_0 - A^{-1} (R(x,y_1) - R(x,y_0))| \\
                &= \left|  \left(\int_0^1 A^{-1} (A - D_y R(x,y_t)) dt\right) (y_1-y_0) \right| \\
                &\leq c |y_1 - y_0|_\infty.
    \end{align*}
    Thus, $F(x,\cdot)$ is indeed a contraction on $\{x\} \times (-\sigma\rho,\sigma\rho)^n$. Hence, the implicit function theory provides $g: (-\sigma\rho,\sigma\rho) \to (-\sigma\rho,\sigma\rho)^n$ with $R(x,g(x)) = 0$.
    
    To estimate the derivatives of the function $g$ effectively, we will proceed as follows. Let
    \begin{align*}
        B(x,y) = (\mathds 1 - A^{-1}(A - D_y R(x,y)))^{-1}.
    \end{align*}
    If $|A^{-1}(A - D_y R))^{-1}|_{\infty \to \infty} < \frac 12$, then $|B|_{\infty \to \infty} < 2$. Note that
    \begin{align*}
        (D_y R)^{-1} = (\mathds 1 - A^{-1}(A - D_y R))^{-1} A^{-1} = B A^{-1}.
    \end{align*}
    Denote $\max(\mu_k^{-1} E_k) \rho^{-1}$ by $\kappa$. Necessarily, $\kappa \geq 1$.
    The chain rule yields $\partial_x g = - (D_y R)^{-1} \partial_x R$. Since $\partial_x R(0,0) = 0$, we can estimate $|A^{-1} \partial_x R| \leq 2 \rho^{-2} \max(\mu_k^{-1} E_k) \max(|x|,|y|)$, and because $\max(|x|,|y|) \leq \sigma \rho$, we find that $|A^{-1} \partial_x R| \leq 2 \sigma \kappa$. Since we chose $\sigma \leq \frac c2 \kappa^{-1}$ and $c \leq \frac 12$, we get $|A^{-1} \partial_x R| \leq c$, and hence $|\partial_x g| \leq |B|\cdot|A^{-1} \partial_x R| \leq 2c \leq 1$.
    Next,
    \begin{align*}
        \partial_x^2 g &= - (D_y R)^{-1} \left( \partial_x^2 R  +  2 \partial_x D_y R \, \partial_x g + D_y^2 R \, \partial_x g \, \partial_x g\right) \\
        |\partial_x^2 g| &\leq |B| \cdot |A^{-1}(2 (\partial_x \partial_{y_{i}} R) \partial_{x} g_i + (\partial_{y_i} \partial_{y_{j}} R) \partial_x g_i \partial_x g_{j} + \partial_x^2 R)| \\
        &\leq 2 \cdot \frac{2}{\rho^2} \cdot \max_{k=1,\ldots,n} \left(\mu_k^{-1} \sum_{|\alpha| = 2} \frac{\rho^2}{\alpha!} |\partial^\alpha R^k| \right) \leq 4 \rho^{-2} \max(\mu_k^{-1} E_k) \leq 4 \kappa \rho^{-1}
    \end{align*}
    Finally,
    \begin{align*}
        \partial_x^3 g = - (D_y R)^{-1} &\big( \partial_x^3 R  +  3 \partial_x^2 D_y R \, \partial_x g +  3 \partial_x D_y^2 R \, \partial_x g \, \partial_x g + D_y^3 R \, \partial_x g \, \partial_x g \, \partial_x g \\
        &+ 3 \partial_x D_y R \, \partial_x^2 g + 3 D_y^2 R \, \partial_x^2 g \, \partial_x g \big)
    \end{align*}
    which implies
    \begin{align*}
        |\partial_x^3 g| &\leq 2 \cdot \frac{6}{\rho^3}  \cdot \max_{k=1,\ldots,n} \left( \mu_k^{-1} \sum_{|\alpha| \leq 3} \frac{\rho^{|\alpha|}}{\alpha!} \partial^\alpha R_k \right) \cdot 4 \kappa \leq 48 \kappa^2 \rho^{-2}.
    \end{align*}
    Now, 
    \begin{align*}
        \partial_x( F \circ (x,g(x))) = (&\partial_x F + \partial_{y_j} F \partial_x g_j) \circ (x,g(x))\\
        \partial_x^2( F \circ (x,g(x))) = (&\partial_x^2 F + 2 \partial_x \partial_{y_j} F \partial_x g_j + \partial_{y_k} \partial_{y_j} F \partial_x g_j \partial_x g_k + \partial_{y_j} F \partial_x^2 g_j)  \circ (x,g(x))\\
        \partial_x^3( F \circ (x,g(x))) = (&\partial_x^3 F + 3 \partial_x^2 \partial_{y_j} F \partial_x g_j + 3 \partial_x \partial_{y_k} \partial_{y_j} F \partial_x g_j \partial_x g_k
        + \partial_{y_\ell} \partial_{y_k} \partial_{y_j} F \partial_x g_\ell \partial_x g_j \partial_x g_k \\
        &+ 3 \partial_x \partial_{y_j} F \partial_x^2 g_j + 3 \partial_{y_k} \partial_{y_j} F \partial_x^2 g_j \partial_x g_k + \partial_{y_j} F \partial_x^3 g_j) \circ (x,g(x)).
    \end{align*}
    Counting combinatorial coefficients and plugging in the estimates on $|\partial_x g|$, $|\partial_x^2 g|$, $|\partial_x^3 g|$ yields
    \begin{align*}
         |\partial_x^3 ( F \circ (x,g(x)))| \leq \frac6{\rho^3} \cdot \left( \sum_{|\alpha| \leq 3} \frac{\rho^{|\alpha|}}{\alpha!} |\partial^\alpha F| \right) \cdot 8 \kappa^2
    \end{align*}

    Next, we specialize the preceding computations to the situation at hand, where $n=2$, the numbers $E_1$ and $E_2$ are bounded above by \eqref{eq:bound_on_eigenvalue_derivatives}--\eqref{eq:bound_on_SP_derivatives} and the numbers $\mu_1$ and $\mu_2$ are given by \eqref{eq:derivative_expressions1}--\eqref{eq:derivative_expressions4}. From \eqref{eq:bound_on_eigenvalue_derivatives}--\eqref{eq:bound_on_Helmholtz_derivatives} we obtain $E_1 \leq 2\tau$ and $E_2 \leq 4\|J_N(\sqrt{\lambda_0} (\cdot) )\|_{L^2(\mathbb D')}^{-1}$, which can be further reduced to $E_2 \leq 4\sqrt{2} \lambda_0^{\frac14}$ by \Cref{lem:bound_on_L2norms}. From \eqref{eq:derivative_expressions1}--\eqref{eq:derivative_expressions5} we obtain $\mu_1 = 6 \lambda_0^2$ and 
    \begin{align*}
        \mu_2 = \left| \frac{\sqrt{\lambda_0}}{\sqrt{\pi}} J_0(\sqrt{\lambda_0})^{-1} \left( J_0(\sqrt{\lambda_0})^{-1} J_0'(\sqrt{\lambda_0}) - I_0(\sqrt{\lambda_0})^{-1} I_0'(\sqrt{\lambda_0}) \right)^{-1} \right|
    \end{align*}
    By classical bounds on Bessel functions, $J_0(\sqrt{\lambda_0}) \leq \sqrt{\frac{2}{\pi}}\lambda_0^{-\frac14}$, and by \Cref{lem:nondegenerate_eigenvalues}, $|W_0(\sqrt{\lambda_0})| \leq 6$. Thus, we obtain
    $\mu_2 \geq \frac{1}{6\sqrt{2}} \lambda_0^{\frac34}$. Thus,
    \begin{align*}
        \rho \min\left(\mu_1 E_1^{-1}, \mu_2 E_2^{-1}\right) &\geq 2\cdot 10^{-7} N^{-2} \frac{\tau}{\lambda_0^2} \min\left(\frac{6\lambda_0^2}{2\tau}, \frac{\lambda_0^{\frac12}}{48}\right) \\
        &\geq 2\cdot 10^{-7} N^{-2} \min\left(3, \frac{\tau}{48\lambda_0^{\frac32}}\right) \\
        &\geq 2 \cdot 10^{-7} N^{-2} \frac{4N^3}{48 \cdot (N+3N^{\frac13})^3} \geq 10^{-8} N^{-2}.
    \end{align*}

    Plugging in the estimates on $\rho, \kappa$ and on $\sum_{|\alpha| \leq 3} \frac{\rho^\alpha}{\alpha!} |\partial^\alpha F|$ yields
    \begin{align*}
        |\partial_t^3 ( F \circ (t,y(t)))| &\leq 192 \kappa^2 \rho^{-3} \| I_0(\sqrt{\lambda_0}(\cdot))\|_{L^2(\mathbb D')}^{-1} \leq 192 \cdot 10^{40} \cdot N^{13} \, \| I_0(\sqrt{\lambda_0}(\cdot))\|_{L^2(\mathbb D')}^{-1}.
    \end{align*}
    By \Cref{lem:bound_on_L2norms}, this bound can be simplified to $192 \cdot 10^{40} N^{13} \cdot 0.38^{-1} \lambda_0^{\frac14} I_0(\sqrt{\lambda_0})^{-1}$.
    Assuming $N \geq 100$ and using \Cref{lem:nondegenerate_eigenvalues} to estimate $\lambda_0^{\frac12} \leq N(1 + 3 N^{-2/3})$, we find that $|\partial_t^3 ( F \circ (t,y(t)))| \leq 600 \cdot 10^{40} N^{13.5}$ as claimed.
\end{proof}

In the subsequent section, only the one-parameter family $\tilde u_s$ and the derivative estimates obtained in \Cref{lem:oneParameterFamily} will be relevant. Therefore, shall denote $(\tilde u_s)_{s \in (-\varepsilon,\varepsilon)}$ by $(u_t)_{t \in (-\varepsilon,\varepsilon)}$, since there is no danger of confusing this one-parameter family with the auxiliary three-parameter family considered in this section.

\section{Proof of Theorem 1: Conclusion}
\label{sec:final}

Here we find the perturbation radius that guarantees there is no significant contribution from third derivatives and show there is also no significant contribution from higher terms in the Bessel series, thus completing the proof of \Cref{thm:main}.

\begin{proof}[Proof of \Cref{thm:main}]
Consider the one-parameter family of domains $D_t$ and $L^2$-normalized eigenfunctions $u_t$ constructed in \Cref{lem:oneParameterFamily}.
Since the diffeomorphisms $\Phi_t$ reduce to the identity on $\mathbb D' = \mathbb D_{1-\frac1N}$, each domain $D_t$ contains $\mathbb D'$. 
The Helmholtz- and SP-components $v_t$ and $w_t$ are orthogonal in $L^2(D_t)$ and $\| u_t \|_{L^2(D_t)} = 1$, which in particular implies $\|v_t\|_{L^2(\mathbb D')} \leq 1$ and $\|w_t\|_{L^2(\mathbb D')} \leq 1$. Both $v_t$ and $w_t$ are $D_N$-invariant by construction, which implies their expansions on the disk $\mathbb D'$ into Bessel and modified Bessel functions take the form
\begin{align*}
    v_t(r,\theta) = \sum_{k=0}^N a_{kN}(t) J_{kN}(\sqrt{\lambda} r) \cos(N\theta), \ w_t(r,\theta) = \sum_{k=0}^N b_{kN}(t) I_{kN}(\sqrt{\lambda} r) \cos(N\theta),
\end{align*}
respectively. Finally, $v_t(0) = 0$ by construction, hence $a_{0} = 0$. Thus, \Cref{lem:remainder_bounds} yields 
\begin{align*}
    |v_t(r,\theta) - a_N(t) J_N(\sqrt{\lambda} r) \cos(N \theta)| &\leq 10 N \exp\left(2N \rho\left(\frac{r}{1-\frac1N}\right)\right) \left( 1 - \exp\left(N \rho\left(\frac{r}{1-\frac1N}\right)\right)\right)^{-2} \\
    |w_t(r,\theta) - b_0(t) I_0(\sqrt{\lambda} r)| &\leq 2\sqrt{N} \left( \frac{r}{1-\frac1N} \right)^N \left( 1 - \left( \frac{r}{1-\frac1N} \right)^N \right)^{-2}
\end{align*}
If $N \geq 100$ and $r < \frac12$, it is easy to see these remainder terms are vanishingly small compared to the corresponding main terms. With some straightforward work, we obtain the following explicit bounds:
\begin{align*}
    |v_t(r,\theta) - a_N(t) J_N(\sqrt{\lambda} r) \cos(N \theta)| &\leq 10^{-16} \exp\left(N \varrho\left(\frac{\sqrt{\lambda}}{N} r\right) \right) \\
    |w_t(r,\theta) - b_0(t) I_0(\sqrt{\lambda} r)| &\leq 10^{-9} \exp\left(N \varrho\left(\frac{\sqrt{\lambda}}{N} r\right) \right)
\end{align*}
Combining the expression \Cref{lem:oneParameterFamily}, we obtain 
\begin{align}
    \label{eq:lower_bound_on_w0}
    w_t(0) &\geq \frac{t^2}{2} \frac{\lambda}{\sqrt{\pi}} I_0(\sqrt{\lambda})^{-1} - \frac{t^3}{6} \cdot 600 \cdot 10^{40} N^{13.5} I_0(\sqrt{\lambda})^{-1} \\
    &\geq \frac{t^2}{6} N^2 I_0(\sqrt{\lambda})^{-1} \left( \frac3{\sqrt{\pi}} - 600 \cdot 10^{40} N^{11.5} t \right).
\end{align}
Choosing $t \leq 10^{-43} N^{-11.5}$ is thus sufficient to ensure $w_t(0) \geq \frac{t^2}{6} N^2 I_0(\sqrt{\lambda})^{-1}$. In the following, we let $t \leq 10^{-43} N^{-11.5 - K_N}$. The exponent $K_N$ should be thought of as an arbitrarily slowly increasing sequence approaching infinity with $N$. Its purpose is merely to ensure that the perturbed domains $D_t$ approach the round disk in the smooth topology as $N \to \infty$ as claimed in the theorem statement. It turns out that if $K_N$ grows reasonably slowly, it has no significant impact on the radius of the circle free of nodal lines.

We begin by obtaining a lower bound on $w_t(r,\theta)$ by combining \eqref{eq:lower_bound_on_w0} and the bounds \Cref{lem:remainder_bounds}(4) and \Cref{lem:remainder_bounds}(3):
\begin{align*}
    w_t(r,\theta) &\geq \frac{t^2}{6} N^2 e^{\sqrt{\lambda}(r-1)} - 10^{-9} \exp\left(N \varrho\left(\frac{\sqrt{\lambda}}{N} r\right) \right) \\
    &\geq 10^{-86} N^{-21 - 2K_N} e^{\sqrt{\lambda}(r-1)} - 10^{-9} \exp\left(N \varrho\left(\frac{\sqrt{\lambda}}{N} r\right) \right).
\end{align*}
On the other hand,
\begin{align*}
    |v_t(r,\theta)| &\leq (5N+10^{-16}) \exp\left(N \varrho\left(\frac{\sqrt{\lambda}}{N} r\right) \right)
\end{align*}
Thus, $|w_t(r,\theta)| > |v_t(r,\theta)|$ (and thus $u_t(r,\theta) > 0$) if 
\begin{align}
    \label{eq:condition_for_positivity}
    \varrho\left(\frac{\sqrt{\lambda}}{N} r\right) < \frac{\sqrt{\lambda}}{N}(r-1) - \frac1N \left( 87 \log 10 + (22 + 2K_N) \log N \right)
\end{align}
Let $\zeta_\infty = 0.443\dots$ denote the unique zero of $s \mapsto (\varrho(s) - s + 1)$ in $(0,1)$. By an elementary calculation, one may check that with $\sigma = \sqrt{1-\zeta_\infty^2} - \zeta_\infty = 0.452\dots$, one has
\begin{align*}
    \sigma \log\left(\frac{s}{\zeta_\infty}\right) \geq \varrho(s) - s + 1,
\end{align*}
and hence, that \eqref{eq:condition_for_positivity} is satisfied if
\begin{align*}
    \frac{\sqrt{\lambda}}{N} r \leq \zeta_\infty \exp\left(- \frac{1}{\sigma} \left(\frac1N \left( 87 \log 10 - (22 + 2 K_N) \log N \right) \right) \right)
\end{align*}
Since $\frac{\sqrt{\lambda}}{N} < 1 + 3 N^{-\frac23} < \exp(3 N^{-\frac23})$ by \Cref{lem:nondegenerate_eigenvalues}, the above is implied by the stronger, but simpler assumption
\begin{align*}
    r \leq \zeta_\infty \exp\left(- 3 N^{-\frac23} - \frac1\sigma \left(\frac1N \left( 87 \log 10 + (22 + 2K_N) \log N \right) \right)\right).
\end{align*}
We may choose $K_N$ growing slowly enough so $\frac2{\sigma} N^{-1} K_N \log N \leq N^{-\frac23}$. Then, the above is implied by the condition 
$$r < \zeta_\infty \exp(-4 N^{-\frac23} - (500 + 50 \log N) N^{-1}),
$$ 
which is therefore a lower bound for the radius of the circle free of nodal lines. 
\end{proof}

\section{Proof of sharpness}
\label{sec:sharpness}

To show the radius $r_\infty$ of the nodal void in \Cref{thm:main} is asymptotically optimal, we need the following lemma, which provides the sharp rate of decay for the frequency of solutions to the Helmholtz equation.

\begin{lemma}[Decio--Malinnikova--Nazarov]
    \label{lem:DMN}
Let $\varepsilon > 0$. Suppose $u \in C^2(\mathbb D)$ is a solution to $\Delta u + \lambda V u = 0$, where $\|V - 1 \|_{C^1(\mathbb D)} < \varepsilon$. Suppose $u$ vanishes to order $k \geq \sqrt{\lambda(1+\varepsilon)}$ at the origin. Denote $I(r) = \int_{B(x,r)} |u|^2$. Then for all $r \in (0,1)$,
\begin{align*}
    \left( \frac{r I'(r)}{I(r)} \right)^2 \geq \sqrt{4(k + 1)^2 - 4\lambda(1+3\varepsilon) r^2}.
\end{align*}
\end{lemma}

Its proof is provided in Appendix~\ref{sec:DMN}. This lemma implies sharp bounds on the rate of decay, in terms of distance to the boundary, of solutions to the Helmholtz equation on a roughly circular domain $D \subseteq \mathbb R^2$ that are not too concentrated on the boundary.

\begin{lemma}
    \label{lem:lowerbound}
There exists $C > 0$ such that for any $\varepsilon \in (0,\frac1C)$, $\sqrt{\lambda} > C \frac{|\log\varepsilon|}{\varepsilon}$, the following holds: Let $\Phi: \mathbb D \to \mathbb R^2$ with $\lVert \Phi - \mathrm{id} \rVert_{C^3(\mathbb D)} < \varepsilon$. Denote $D = \Phi(\mathbb D)$. Let $v \in C^2(\overline D)$ be a solution to $\Delta v + \lambda v = 0$ on $D$ and assume $\int_{\partial D} |v|^2 < 10$ and $\int_{D} |v|^2 > \frac1{10}$. Then
\begin{align*}
    \int_{\mathbb D_r} |v|^2 > \exp(2 \sqrt{\lambda}\varrho\left(r/(1+C\varepsilon)\right))
\end{align*}
holds for any $r \in (0,1-\frac{1}{100})$.
\end{lemma}

\begin{proof}
    In the following, purely for the sake of brevity, $C$ is an absolute constant that may be replaced by a larger absolute constant each time it occurs. Let $k = \lceil (1+C\varepsilon)\sqrt{\lambda} \rceil$ and write $v = v_{<} + v_{\geq}$, with 
    \begin{align}
    \label{eq:low_components}
        v_{<}(r,\theta) = \sum_{\ell=0}^n a_\ell J_\ell(\sqrt{\lambda} r) \cos(\ell \theta - \eta_\ell) 
    \end{align}
    the unique Bessel series agreeing with $v$ up to $(k-1)^{st}$ order and $v_{\geq} = v - v_{<}$.

    It follows from $L^2$-orthogonality on $\mathbb D_r$ of the summands in \eqref{eq:low_components}, the bounds of \Cref{lem:bounds} on the growth of individual Bessel functions and an application of the triangle inequality that
    \begin{align*}
        |v_<(s,\theta)| < \exp(- \sqrt{\lambda} \varrho(r/(1+C\varepsilon)) + C \log \lambda) \| v_{<} \|_{L^2(\mathbb D_r)}.
    \end{align*}
    for any $s \in (1-C\varepsilon,1+C\varepsilon)$. Since the derivative of $\varrho$ is bounded below on $(0,1-\frac1{100})$, and by assumption $\frac{\varepsilon}{C}\sqrt{\lambda} \geq C \log \lambda$, we may absorb $C\log \lambda$ by enlarging $C$, obtaining
    \begin{align*}
        |v_<(s,\theta)| < \exp(- \sqrt{\lambda} \varrho(r/(1+C\varepsilon))) \| v_{<} \|_{L^2(\mathbb D_r)}.
    \end{align*}
    
    Now assume for the sake of contradiction that, for some fixed $C' > C$, we have
    \begin{align*}
        \int_{\mathbb D_r} |v_<|^2 < \exp(2 \sqrt{\lambda}\varrho\left(r/(1+C'\varepsilon)r\right)).
    \end{align*}
    Then $\int_{\partial D} |v_<|^2 < \exp(- \delta \sqrt{\lambda} )$ and $\int_{D} |v_<|^2 < \exp(- \delta \sqrt{\lambda} )$. Note that $(v,w)_{L^2(D)} = 0$ and $\|v\|_{L^2(D)} \geq \|w\|_{L^2(D)}$ and hence $\int_{D} |v|^2 > \frac{1}{20}$. Thus, $\int_{\partial D} |v_\geq|^2 < 11$ and $\int_{D} |v_\geq|^2 > \frac{1}{21}$.

    It follows from the proof of the classical Kellogg--Warschawski theorem on boundary regularity of the Riemann mapping that there exists a conformal map $\phi: \mathbb D \to D$ with $\phi(0) = 0$ and $\|\phi - \mathrm{id}\|_{C^{2}(\mathbb D)} \leq C \|\Phi - \mathrm{id}\|_{C^{3}(\mathbb D)}$ for absolute constant $C > 0$ (the seeming loss of regularity arises purely out of convenience, in fact $\|\phi - \mathrm{id}\|_{C^{2,\alpha}(\mathbb D)} \leq C \|\Phi - \mathrm{id}\|_{C^{2,\alpha}(\mathbb D)}$ for any $\alpha \in (0,1)$ and an $\alpha$-dependent $C > 0$). Let $\tilde v_{\geq} = \phi^\ast v_{\geq}$. We have $\Delta \tilde v_{\geq} + \lambda V u = 0$, with $V = |\partial_z \phi|^2$ satisfying $\|V - 1\|_{C^1(\mathbb D)} < C \varepsilon$. Furthermore, $\int_{\partial \mathbb D} |\tilde v_\geq|^2 < 12$ and $\int_{\mathbb D} |\tilde v_\geq|^2 > \frac{1}{22}$. Let $I(r) = \int_{\mathbb D_r} |\tilde v_{\geq}|^2$ as in \Cref{lem:DMN}. Then
    \begin{align*}
        12 \cdot 22 \geq \left( \frac{I'(1)}{I(1)} \right)^2 \geq \sqrt{4(k + 1)^2 - 4\lambda(1+3\varepsilon)},
    \end{align*}
    and since $\sqrt{\lambda} \geq \frac{C}{\varepsilon}$, it follows that $v_{\geq}$ vanishes to order $k \leq \sqrt{(1+4 \varepsilon)\lambda}$ at the origin, a contradiction.
\end{proof}

\begin{proof}[Proof of \Cref{thm:sharpness}]
    As in \Cref{sec:proof}, we write $u = w - v$, where $v = \frac{1}{2\lambda}(\Delta u - \lambda u)$ and $w = \frac{1}{2\lambda}(\Delta u + \lambda u)$. Since $\|u\|_{L^2(D)}^2 = 1$,
    \begin{align*}
        \int_{D} vw &= \frac{1}{4\lambda^2} \int_{D} (|\Delta u|^2 - \lambda^2 u^2) = 0, \\
        \int_{D} (w-v)(w+v) &= \frac{1}{\lambda} \int_{D} u \Delta u = - \frac{1}{\lambda} \int_{D} |\nabla u|^2 < 0,
    \end{align*}
    we conclude that $1 \geq \|v\|_{L^2(D)}^2 \geq \frac12$. Let $X = r \partial_r$, let $\phi_t$ denote the flow generated by $X$, write $D_t = \phi_t(D)$ and consider the corresponding eigenvalue and eigenfunction $\lambda_t$ and $u_t$, respectively, as given by Rellich's one-parameter perturbation theorem.
    By the scaling invariance of $\Delta^2$, $\lambda_t = e^{-2 t} \lambda$, so $\dot \lambda = - 2 \lambda$. On the other hand, the eigenvalue variational formula yields
    \begin{align*}
        2 \lambda \dot \lambda = - \int_{\partial D} (X \cdot n) |\Delta u|^2 = - 4 \lambda^2 \int_{\partial D} |v|^2,
    \end{align*}
    hence $1 = \int_{\partial D} (X \cdot n)|v|^2$, which implies $1 - C \varepsilon < \int_{\partial D} |v|^2 < 1 + C \varepsilon$, since $\| \Phi - \mathrm{id} \|_{C^1} < \varepsilon$ and hence $1 - C \varepsilon < (X \cdot n) < 1 + C \varepsilon$. \footnote{This argument extends to any star-shape domain $D$. In an upcoming article, Kuperman, Lin, Logunov and Mangoubi will demonstrate a similar bound holds indeed for any smoothly bounded domain $D$ \cite{Zhengjiang}.}
    
    Thus, \Cref{lem:lowerbound} applies to $v$ and yields 
    \begin{align*}
        \int_{\mathbb D_r} |v|^2 > \exp(2 \sqrt{\lambda}\varrho\left(r/(1+C\varepsilon)\right))
    \end{align*}
    for any $r \in (0,1)$. On the other hand, the function $\tilde w = w/I_0(\sqrt{\lambda}|\cdot|)$ satisfies the maximum principle, and since $\|w\|_{L^2(\mathbb D_{1-C\varepsilon})}^2 < 1$ and $I_0(\sqrt{\lambda} r) \leq e^{r/(1-C\varepsilon)}$, it follows that
\begin{align*}
    |w(r,\theta)| \leq \exp(-\sqrt{\lambda} (1-r/(1-C\varepsilon))).
\end{align*}

Basic elliptic theory yields the reverse Hölder inequality $\|v\|_{L^2(D_{s + \varepsilon})} \leq C \lambda^{\frac12} \|v\|_{L^1(D_s)}$. Indeed, Caccioppoli's inequality yields $\|v\|_{H^2(D)}^2 \lesssim \max(\varepsilon^{-2},\lambda) \|v\|_{L^2(D)}^2$. By assumption, $\sqrt{\lambda} \geq \varepsilon^{-1}$. Now the Sobolev embedding theorem $\|v\|_{L^4(D)} \lesssim \|v\|_{H^1(D)}$ and an application of Cauchy's inequality yields the claim.

It follows that for any $r \in (\frac13,\frac23)$ we have
\begin{align*}
    \int_{\mathbb D_r} |v| > \exp(\sqrt{\lambda}\varrho\left(r/(1+C \varepsilon)r\right)).
\end{align*}
Now choose $r \in (r_\infty, (1+C\varepsilon) r_\infty)$ such that $J_0'(\sqrt{\lambda} r) = 0$ -- since such values of $r$ are spaced at distance approximately $\pi/\sqrt{\lambda}$, this is possible. Then $v$ has zero average over $\mathbb D_r$, as indeed does any solution to the Helmholtz equation on $\mathbb D_r$. As a consequence,
\begin{align*}
    \int_{\mathbb D_r} |\max(-v,0)|, \int_{\mathbb D_r} |\max(v,0)| > \frac12\exp(\sqrt{\lambda}\varrho\left(r/(1+C \varepsilon)r\right)).
\end{align*}

It follows that $\max v \geq \exp(\sqrt{\lambda}\varrho\left(r/(1+C \varepsilon)r\right))$ and $|\min v| \geq \exp(\sqrt{\lambda}\varrho\left(r/(1+C \varepsilon)r\right))$, so $u = w-v$ changes sign on $\mathbb D_r$ if $\varrho\left(r/(1+C\varepsilon)r\right) > - (1-C\varepsilon - r)$. Since $\varrho(r)+1-r$ has a simple zero with nonvanishing derivative at $r = r_\infty$, this holds if $r \geq (1 + C \varepsilon)r_\infty$.
\end{proof}

\appendix

\section{Proof of \Cref{lem:DMN}}
\label{sec:DMN}

The authors are grateful to Eugenia Malinnikova for communicating the following proof of \Cref{lem:DMN}, which is reproduced here for completeness' sake.

\begin{proof}[Proof of \Cref{lem:DMN}]
    This proof works in the $d$-dimensional unit ball $\mathbb B^d$, and so shall be presented in this generality. We begin by denoting
\begin{align}
        I(r) &= \int_{B_r} |u|^2, & H(r) &= \int_{\partial B_r} |u|^2, \\
        E(r) &= \int_{\partial B_r} u \partial_n u, & J(r) &= \int_{B_r} V |u|^2
\end{align}
By the divergence theorem, $E(r) = \int_{B_r} (|\nabla u|^2 - \lambda V |u|^2)$. Of course,
\begin{align*}
    I'(r) &= H(r), \\
    H'(r) &= \frac{d-1}{r} H(r) + 2 E(r).
\end{align*}
Now we estimate $E'(r)$. We decompose $E'(r) = \int_{\partial B_r} |\nabla u|^2 - \lambda \int_{\partial B_r} V |u|^2$ into its two constituents and estimate them separately. Write $X = r \partial_r$. Pohozaev's identity for this vector field reads
\begin{align*}
    &\int_{\partial B_r} (X \cdot \nabla u) \partial_n u = \int_{B_r} \divg((X \cdot \nabla u) \nabla u) \\
    &\ = \int_{B_r} (X \cdot \nabla u) \Delta u + \int_{B_r} DX(\nabla u, \nabla u) + \frac12 \int_{B_r} X \cdot \nabla( |\nabla u|^2 ) \\
    &\ = \int_{B_r} (X \cdot \nabla u) \Delta u + \int_{B_r} DX(\nabla u, \nabla u) + \frac12 \int_{\partial B_r} (X \cdot n) |\nabla u|^2 - \frac12 \int_{B_r} \divg X |\nabla u|^2
\end{align*}
Now, since $(X \cdot n) = r$, $\divg X = d$ and $DX = \mathds{1}$, we find 
\begin{align*}
    \int_{\partial B_r} |\nabla u|^2 = 2 \int_{\partial B_r} |\partial_n u|^2 + \frac{d-2}{r} \int_{B_r} |\nabla u|^2 + \frac{2\lambda}{r} \int_{B_r} Vu (X \cdot \nabla u)
\end{align*}
We use the divergence theorem again to obtain
\begin{align*}
    &\int_{\partial B_r} (X \cdot n) V |u|^2 = \int_{B_r} (\divg X)  V |u|^2 + \int_{B_r} (X \cdot \nabla V) |u|^2 + 2 \int_{B_r} V u (X \cdot \nabla u),
\end{align*}
and hence,
\begin{align*}
    \lambda \int_{\partial B_r} V |u|^2 = \frac{\lambda d}{r} \int_{B_r} V |u|^2 + \frac{2\lambda }{r} \int_{B_r} V u (X \cdot \nabla u) + \varepsilon \frac{\lambda}{r} \int_{B_r} |u|^2.
\end{align*}
The term $\int_{B_r} V u (X \cdot \nabla u)$ cancels and we find 
\begin{align*}
    E'(r) &= 2 \int_{\partial B_r} |\partial_n u|^2 + \frac{d-2}{r} E(r) - \frac{2\lambda}{r} J(r) - \varepsilon \frac{\lambda}{r} I(r) \\
    &\geq \frac{2 E(r)^2}{H(r)} + \frac{d-2}{r} E(r) - (2+3\varepsilon) \frac{\lambda}{r} I(r).
\end{align*}
Next, we introduce the appropriate equivalent of the frequency function for this problem,
\begin{align*}
    N(r) = \frac{r E(r)}{H(r)}.
\end{align*}
It satisfies the differential inequality
\begin{align*}
    N'(r) &= N(r) \left( \frac{1}{r} + \frac{E'(r)}{E(r)} - \frac{H'(r)}{H(r)}\right) \\
    &\geq N(r) \left( \frac1r + 2 \frac{E(r)}{H(r)} + \frac{d-2}{r} - (2 + 3\varepsilon) \frac{\lambda}{r} \frac{I(r)}{E(r)} - \frac{d-1}{r} - 2 \frac{E(r)}{H(r)}\right) \\
    &\geq - N(r) (2 + 3\varepsilon) \frac{\lambda I(r) }{r E(r)} = - (2 + 3\varepsilon) \lambda \frac{I(r)}{H(r)}
\end{align*}
Write $\Phi(r) = \frac{I(r)}{H(r)}$. Then
\begin{align*}
    \Phi'(r) = 1 - \frac{H'(r)}{H(r)} \Phi(r) = 1 - \frac{1}{r} \left( d - 1 + 2 N(r) \right) \Phi(r).
\end{align*}
These two inequalities suffice to sharply control the decay of $N(r)$ as $r$ increases, as the following lemma shows.

\begin{lemma}
    \label{lem:DMN_diffineq}
    Let $0 < r_1 < r_2$. Let $F,G \in C^1([r_1,r_2])$ and $b \in C^0([r_1,r_2])$ be positive functions. Assume 
    \begin{align*}
        F'(r) &\leq 1-\frac1r F(r)G(r) \\
        G'(r) &\geq -b(r) F(r)
    \end{align*}
    for all $r \in [r_1,r_2]$ and suppose that $F(r_1)G(r_1) < r_1$. Then 
    \begin{align*}
        G(r)^2 \geq G(0)^2 - 2 \int_{r_1}^r \rho b(\rho) d\rho.
    \end{align*}
    for all $r \in [r_1,r_2]$.
\end{lemma}

\begin{proof}
    Let $\hat G(r) = \min_{r' \in (0,r)} G(r')$. We first claim that $F(r)\hat G(r) \leq r$. If not, there would be a smallest value $r^\ast$ for which $F(r^\ast)\hat G(r^\ast) = r^\ast$. But since $\hat G$ is non-increasing, this would imply $F'(r^\ast) \geq r^\ast / \hat G(r^\ast)$, which contradicts $F'(r^\ast) \leq 1 -\frac1{r^\ast} F(r^\ast)G(r^\ast) = 0$.

    Note that $\hat G'(r)$ (defined to be lower semicontinuous) is either $0$ or $G'(r)$, and hence, $\hat G'(r) \geq - b(r) F(r)$. It follows that 
    \begin{align*}
        \left(\tfrac12 \hat G(r)\right)' \geq - r b(r),
    \end{align*}
    and thus $\hat G(r)^2 \geq G(0)^2 - 2\int_0^r \rho b(\rho) d\rho$ as claimed.
\end{proof}

Suppose now that $u$ vanishes to order $k \geq \sqrt{\lambda(1+3\varepsilon)}$ at the origin. As $r \to 0$, $\Phi(r) \to \frac{r}{2k + d}$, while $N(r) \to k$, hence $2N(r) + d - 1 \to 2k + d - 1$. Thus, the hypotheses of \Cref{lem:DMN_diffineq} are satisfied, and
\begin{align*}
    \frac{r H(r)}{I(r)} \geq \sqrt{(2k + d - 1)^2 - 4 \lambda (1 + 3\varepsilon) r^2 },
\end{align*}
proving \Cref{lem:DMN}. \end{proof}

\section*{Acknowledgements}

This work has received funding from the European Research Council (ERC) under the European Union's Horizon 2020 research and innovation programme through the grant agreement~862342 (A.E.). It is also partially supported by the MCIN/AEI grants CEX2023-001347-S, RED2022-134301-T and PID2022-136795NB-I00 (A.E.). The authors thank Iosif Polterovich for an informative discussion on the background of the problem of interest. The second author is grateful to Eugenia Malinnikova for her untiring support and expert guidance, as well as many helpful comments on this paper.

\end{document}